\documentclass[a4paper,12pt]{amsart}

\usepackage{amssymb,amsbsy,amsmath,amsfonts,amssymb,amscd}
\usepackage{latexsym}
\usepackage{graphics}
\usepackage{color}

\renewcommand{\Sigma}{\Gamma}

\usepackage{tikz}
\usetikzlibrary{decorations.markings}
\usetikzlibrary{arrows}
\usepgflibrary{shapes}
\input xy
\xyoption{all}

\newcommand\sC{{\mathcal C}}

\newcommand\sB{{\mathcal B}}

\newcommand\Ga{\Sigma}

\newcommand\ga{\gamma}
\newcommand\de{\delta}

\def\Bbb{\bf}

\newcommand{\RR}{\ensuremath{\mathbb{R}}}
\newcommand{\ZZ}{\ensuremath{\mathbb{Z}}}

\newcommand{\sS}{\ensuremath{\mathcal{S}}}

\newcommand{\NN}{\ensuremath{\mathbb{N}}}

\newcommand{\HH}{\ensuremath{\mathbb{H}}}
\newcommand{\PP}{\ensuremath{\mathbb{P}}}

\newcommand{\ra}{\ensuremath{\rightarrow}}

\def\eea{\end{eqnarray*}}
\def\bea{\begin{eqnarray*}}

\newcommand\dual{\mathrel{\raise3pt\hbox{$\underline{\mathrm{\thinspace d
\thinspace}}$}}}

\newcommand\qe{\ifhmode\unskip\nobreak\fi\quad $\Box$}

\def\BOX{\hfill\lower.5\baselineskip\hbox{$\Box$}}

\newtheorem{theo}{Theorem}[section]
\newtheorem{remarkk}[theo]{Remark}
\newenvironment{rem}{\begin{remarkk}\rm}{\end{remarkk}}

\newtheorem{defin}[theo]{Definition}
\newenvironment{definition}{\begin{defin}\rm}{\end{defin}}

\newtheorem{prop}[theo] {Proposition}
\newtheorem{cor}[theo]{Corollary}
\newtheorem{lemma}[theo]{Lemma}
\newtheorem{example}[theo]{Example}

\newcommand{\sR}{\ensuremath{\mathcal{R}}}

\newcommand{\ul}{\underline}
\newcommand{\ol}{\overline}

\renewcommand{\a}{\alpha}
\renewcommand{\b}{\beta}
\newcommand{\cg}{\gamma}

\renewcommand{\d}{\delta}

\newcommand{\e}{\varepsilon}

\renewcommand{\l}{\lambda}

\newcommand{\x}{\xi}

\newcommand{\tto}{\longrightarrow}

\newcommand{\sT}{\ensuremath{\mathcal{T}}}

\newcommand{\Dn}{{D}_n}

\newcommand{\cutoff}[1]{}

\DeclareMathOperator{\Aut}{Aut}

\newcommand{\Proof}{{\it Proof. }}

\begin{document}

\title[Dihedral covers of algebraic curves]{The irreducible components of the moduli space of dihedral covers of algebraic curves}\author{Fabrizio Catanese, Michael L\"onne, Fabio Perroni}

\thanks{The present work took place in the realm of the DFG
Forschergruppe 790 `Classification of algebraic
surfaces and compact complex manifolds'.
}

\date{\today}

\maketitle

\begin{abstract}
The main purpose of this paper is to introduce a new 
invariant for the action of a finite group $G$ on a compact complex curve of genus $g$.
With the aid of this invariant we achieve the classification of the components of the locus 
(in the moduli space) of curves admitting an effective action  by the dihedral group $D_n$.
This invariant { has later been used in \cite{CLP13} where}  
 the results of Livingston \cite{Liv} and of  Dunfield and Thurston \cite{DT}
{  have been  extended to the ramified case}.
\end{abstract}


\section{Introduction}
We  study  moduli spaces of curves that admit an effective action 
by a given finite group $G$. These moduli spaces can be seen as closed algebraic subsets $M_g(G)$ of $M_g$,
the moduli space of smooth curves of genus $g>1$.
 We are mainly interested in understanding which are the irreducible components of 
$M_g(G)$. 

To a curve $C$ of genus $g$ with an action by $G$, we can associate several discrete 
invariants that are constant under deformation.

On one hand, the \emph{topological type}
of the $G$-action is a homomorphism 
$\rho \colon G \to Map_g$
well-defined up to inner conjugation induced by
different choices of an isomorphism  $Map (C)  \cong  Map_g$,
see Section 2.

It turns out that the locus $M_{g,\rho}(G)$ of curves admitting a $G$-action of topological type
$\rho$ is a closed irreducible subset of $M_g$, 
(see Theorem \ref{topological action}).

On the other hand
the action of $G$ on $C$ gives rise to a morphism $p\colon C \to C/G=: C'$, a $G$-cover,
 and the geometry of $p$
encodes several numerical invariants that are constant on $M_{g,\rho}(G)$:
the genus $g'$ of $C'$, the number $d$ of branch points $y_1, \dots , y_d \in C'$ and  the
orders $m_1 \leq  \dots \leq  m_d$ of the local monodromies. These numbers $g', d, m_1 \leq  \dots \leq  m_d$  form the {\it primary numerical type}. 

A second numerical invariant is obtained from the monodromy
\linebreak
{$\mu \colon \pi_1(C'\setminus\{ y_1, \dots , y_d\})\to G$}  of the restriction of $p$
to $p^{-1}(C'\setminus \{ y_1, \dots , y_d \})$,
and is called the \emph{$\nu$-type} or Nielsen function.
It is a class function $\nu$ 
which, for each conjugacy class $\mathcal{C}$ in $G$, counts the
number of local monodromies which belong to $\mathcal{C}$.

Observe that  the  irreducible closed algebraic sets $M_{g,\rho}(G)$ 
depend only upon what we call the `unmarked topological type', which is defined
as the conjugacy class of the subgroup  $\rho (G)$  inside $Map_g$.

The following is immediate
by Riemann's existence theorem and the irreducibility of
the moduli space $M_{g',d}$ of $d$-pointed curves of genus
$g'$.
Given $g'$ and $d$, the unmarked topological types
whose primary numerical type is  of the form $g', d, m_1, \dots , m_d$  are in bijection with
the quotient of the set of the corresponding  monodromies $\mu$ modulo 
the actions by $Aut(G)$ and by $Map(g',d)$. \\ 
Here $Map(g',d)$
is the full mapping class group of genus $g'$ and $d$ unordered  points. Thus  a first step toward the general problem 
consists in finding  a fine  invariant that distinguishes these orbits. 

In this paper we introduce a new invariant $\hat{\e}$ for $G$-actions on smooth curves.
 
In the case where $G$ is the dihedral group $D_n$ of order $2n$,
we show that
$\hat{\e}$ distinguishes the different  unmarked topological types, 
and  therefore $\hat{\e}$ is a fine invariant  in the dihedral case. 

Our invariant includes and extends two well known invariants that have been studied in the literature:
the $\nu$-type  (or Nielsen type) of the cover  (also called shape in \cite{FV}, cf. Def. \ref{Nielsen})
and
the class in the second homology group $H_2(G/H, \ZZ)$
(modulo the action of $Aut(G/H)$) corresponding to the unramified cover
$p'\colon C/H \to C'$, where $H$ is the 
minimal
normal subgroup of $G$ generated by the local monodromies.\\
These invariants,  which refine the primary numerical type, provide a fine  invariant under some restrictions, for instance when 
$G$ is abelian  or when $G$ acts freely and is the semi-direct product of two finite cyclic groups 
(as it follows by combining results from  \cite{FabIso}, \cite{cyclic},  \cite{Edm I} and  \cite{Edm II}). 
However, in general, they are not enough to distinguish unmarked topological types,
as one can see already for non-free $D_n$-actions (see Lemma \ref{bad case}). 

The construction of $\hat{\e}$ is similar in spirit to the procedure that,
using Hopf's  theorem, associates an element in $H_2(G,\ZZ)$ to any free $G$-action on a smooth curve.
So much cannot be achieved in the `branched' case of a non- free action. In this case we are only able to associate to
two given 
actions with the same $\nu$-type, an invariant in a quotient group of $H_2(G,\ZZ)$ which is the `difference' of the respective $\hat{\e}$- invariants. 
  Here is  the way we do it. 
 For any finite group $G$, let $F$ be the free group
generated by the elements of $G$ and let $R\trianglelefteq F$ be the subgroup of relations, that is $G=F/R$.
For any $\Sigma \subset G$, union of non trivial conjugacy classes, 
let $G_\Sigma$ be the quotient group of $F$
by the  minimal normal subgroup generated by $[F,R]$ and  by the elements
$\hat{a}\hat{b}\hat{c}^{-1}\hat{b}^{-1}\in F$, for any  $a, c \in \Sigma$, $b\in G$,
such that  $b^{-1} ab=c$. Here we denote by $\hat{g}\in F$ the generator corresponding to $g\in G$.
To a given $G$-cover $p\colon C \to C'$ we associate the set $\Sigma$ of { local monodromies, i.e., of} elements 
which  i) stabilize some point $x$ of $C$
 and ii) act on the tangent space at $x$ by a rotation  of angle
$\frac{2\pi}{m}$ where $m$ is the order of the stabilizer at $x$.
Upon the choice of a geometric basis for the fundamental group 
of the  complement $C'\setminus \{ y_1, \dots , y_d\}$  of the branch set,
 our cover is given by an element $v=(c_1, \dots , c_d; a_1, b_1, \dots , a_{g'}, b_{g'})\in G^{d+2g'}$
satisfying certain conditions (a \emph{Hurwitz generating system}), where the first entries correspond to the local monodromies.  
Thereby $\Sigma =\Sigma_v$ is  the union of the conjugacy classes of the $c_i$'s. 
The tautological lift $\hat{v}$ of $v$ is 
$(\widehat{c_1}, \dots , \widehat{c_d}; \widehat{a_1}, \widehat{b_1}, \dots , \widehat{a_{g'}}, \widehat{b_{g'}})$.
Finally, define $\e(v)$ as the class in $G_{\Sigma}$ of 
$$
\prod_1^d\widehat{c_j} \cdot \prod_1^{g'}[\widehat{a_i}, \widehat{b_i}] \, .
$$
It turns out that  the image of $\e(v)$ in $\left( G_{\Sigma}\right)/_{Inn(G)}$ is invariant under the action of $Map(g',d)$,
as shown in  Proposition \ref{e-inv}. 
Moreover the $\nu$-type of $v$ can be deduced from $\e(v)$, as it is essentially the image of $\e(v)$ in the abelianized group 
$G_{\Sigma}^{ab}$ (see the Remark after Def.\ \ref{Nielsen}).

In order to take into account also the automorphism group $Aut(G)$, we define 
$$
G^\cup : = {\coprod}_\Sigma G_\Sigma \, ,
$$
the disjoint union of all the $G_\Sigma$'s. Now, the group $Aut(G)$ acts on $G^\cup$ and we get a map
$$
\hat{\e}\colon \left(HS(G; g',d) /_{Aut(G)}\right)/_{Map(g',d)} \to (G^\cup)/_{Aut(G)} \, 
$$
which is induced by $v\mapsto \e(v)$. Here we denote by $HS(G; g',d)$ the set of all 
Hurwitz generating systems of length $d+2g'$.

Finally, we prove that
the map $\hat{\e}$ is injective  in the case $G=D_n$ and we determine the image of $\hat{\e}$,
(Theorem \ref{Dn-case}), thus the invariant 
$\hat{\e}$ is a fine invariant for $D_n$-actions. This completes the classification of the 
unmarked topological types for $G = D_n$, begun in \cite{CLP11}.

{ We finally show how this classification entails the classification of the irreducible components of the loci $M_g(D_n)$.  }

When $g'=0$ our $G_\Sigma$ is related to the group $\widehat{G}$ defined in \cite{FV} (Appendix),
where the authors give a proof of a theorem by Conway and Parker. Roughly speaking the theorem says that:
if the Schur multiplier $M(G)$ (which is isomorphic to $H_2(G,\ZZ$))
 is generated by commutators, then the $\nu$-type is a fine stable invariant, when $g'=0$.\\
 Results of this kind, when $g'>0$ but for free $G$-actions and any finite group $G$,
 have been proved in \cite{Liv} and \cite{DT}. This time the fine stable invariant lives in $H_2(G,\ZZ)/_{Aut(G)}$.\\
 The natural question   whether our $\hat{\e}$-invariant is a fine stable invariant for any finite group $G$
 and any effective $G$-action on compact curves
has been solved   in \cite{CLP13} for genus stabilisation.

\medskip

The structure of the paper is the following. In Section 2 we introduce the moduli spaces $M_g(G)$ 
and the subsets $M_{g,\rho}(G)$.
Using Riemann's existence theorem, we reduce the problem of the determination of 
the loci $M_{g,\rho}(G)$ to a combinatorial one. This leads to the concept of {topological type} 
and of Hurwitz generating system. 
In Section 3 we define the 
map $\hat{\e}$, the groups $H_{2,\Sigma}(G)$
and we prove some properties. The object of  Section 4 is the computation of $H_{2,\Sigma}(D_n)$.
These results are all used in Section 5 where we prove the injectivity of $\hat{\e}$ when $G=D_n$. 
In  Appendix { A} we collect some  results  about mapping class groups and their action 
on fundamental groups. We use these results in the proof of Theorem \ref{Dn-case}. 

 Appendix B describes the  case (see especially  theorem \ref{exception}) where two irreducible loci $M_{g, \rho}(D_n)$ 
coincide.

\section{Moduli spaces of $G$-covers}
Throughout this Section $g$ is an integer, $g>1$. The moduli space of curves of genus $g$
is denoted by $M_g$. For any finite group $G$, $M_g(G)$ is the locus of $[C]\in M_g$ such that
there exists an effective action of $G$ on $C$.  For any $[C]\in M_g(G)$, the quotient morphism
$p\colon C \to C/G=C'$ is a Galois cover with group $G$, a $G$-cover,  well defined up to isomorphisms.

Riemann's existence theorem allows us to use combinatorial methods to study  $G$-covers,
since $p$ determines and is determined by its restriction to $C'\setminus \mathcal{B}$,
where $\mathcal{B}=\{\, y_1\, , \, \dots \, , \, y_d\, \}\subset C'$ is the branch locus of $p$. \\
Fix a base point $y_0\in C' \setminus \mathcal{B}$ and a point $x_0\in p^{-1}(y_0)$.
Monodromy  gives a surjective group-homomorphism
\begin{equation}\label{m}
\mu \colon \pi_1 (\, C'\setminus \mathcal{B} \, , \, y_0 \, ) \longrightarrow G \, 
\end{equation}
that characterizes  $p$ up to isomorphism.

Let us  recall that  a {\bf geometric 
basis} of $\pi_1 (\, C'\setminus \mathcal{B} \, , \, y_0 \, )$
consists of simple non-intersecting 
 geometric
loops  based at $y_0$
$$
\gamma_1 , \dots , \gamma_d , \a_1, \b_1, \dots , \a_{g'}, \b_{g'}
$$
such that we get the presentation
\begin{equation*}
\pi_1 (\, C'\setminus \mathcal{B}\, , \, y_0 \, ) = \langle \, \gamma_1 , \, \dots \, ,  \gamma_d\, ; \, \a_1 ,  \b_1 , \, \dots \, ,  \a_{g'} ,  \b_{g'} \, | \, 
\prod_1^d \gamma_j \cdot \prod_{1}^{g'}[ \a_i ,  \b_i ] =1\, \rangle\, .
\end{equation*}

Varying a covering in a flat family with 
connected base, there are some numerical 
invariants which remain
unchanged, the first ones being the respective 
genera $g, g'$ of the curves $C$, $C'$, which are 
related by the
Hurwitz formula:
\begin{equation}\label{Hurwitz}
2 (g-1) = | G | [ 2 (g'-1) + \sum_i ( 1 - 
\frac{1}{m_i})], \ \ m_i : = ord (\mu (\cg_i) ) \, . 
\end{equation}
Observe moreover that a different choice of the 
geometric basis changes the generators
$\cg_i$, but does not change their conjugacy classes (up to permutation),
hence another numerical invariant is provided by the number of elements
$\mu (\cg_i) $ which belong to a fixed conjugacy class in the group $G$.

We formalize these invariants through  the following definition.

\newcommand{\gd}{g'\negmedspace,\!d\,}

\begin{definition}\label{factorisations}
Let $G$ be a finite group, and $g',d \in \NN$. 
A \emph{ $(\gd)$-Hurwitz vector}
in $G$ is an element $v\in G^{d+2g'}$, the Cartesian product of $G$ $(d+2g')$-times. 
A  $(\gd)$-Hurwitz vector in $G$ will  also be denoted by
$$
v=(c_1, \dots , c_d \, ; \, a_1, b_1, \dots , a_{g'},b_{g'})\, .
$$
For any $i\in \{ 1, \dots , d+2g'\}$, the $i$-th component $v_i$ of $v$ is defined as usual.
The \emph{evaluation} of $v$ is the element 
$$
ev(v)=\prod_1^d c_j \cdot \prod_1^{g'}[a_i, b_i] \in G \, .
$$ 

A \emph{Hurwitz generating system of length $d+2g'$ in $G$} is a  $(\gd)$-Hurwitz vector $v$ in $G$
such that the following conditions hold:
\begin{itemize}
\item[(i)] $c_i\not=1$ for all $i$;
\item[(ii)] $G$ is generated by the components  $v_i$ of 
$v$;
\item[(iii)] 
$ev(v) = 1 $.
\end{itemize}
We denote by $HS(G;g',d)\subset G^{d+2g'}$ the set of all Hurwitz generating systems in $G$ of length $d+2g'$.
\end{definition}

\begin{definition}\label{admissible}
The condition $ev(v)=\prod_1^d c_j \cdot \prod_1^{g'}[a_i, b_i] = 1 $ immediately implies that 
the product $\prod_1^d c_j$ has trivial image in the abelianization $G^{ab}$ of $G$.
Observe that the image of $c_j$ inside $G^{ab}$ only depends on the conjugacy class $\sC_j$ of $c_j$.

Denote by $[\sC]$, for each conjugacy class $\sC$ in $G$, its image inside $G^{ab}$.

One defines the Nielsen class function  of $v$ as the function which, on each conjugacy class $\sC$ in $G$, takes
the value $\nu (v) (\sC): = | \{ j | c_j \in \sC \}|$.

 We shall say   that a class function $\nu$ is {\bf admissible} if it satisfies:
 $$ \sum_{\sC}   \nu (\sC)  [\sC] = 0 \in G^{ab}.$$
\end{definition}

Notice that, once we fix a base point $y_0\in C' \setminus \mathcal{B}$ and a geometric basis
of $\pi_1(C'\setminus \mathcal{B}, y_0)$, there is a one-to-one correspondence between
the set of Hurwitz generating systems of length $d+2g'$ in $G$ and the set of monodromies $\mu$ as in \eqref{m}.

\medskip 

\paragraph{\bf Topological type.}

We recall a result contained in \cite{FabIso}, see also \cite{cime}.

Define the orbifold fundamental group 
$\pi_1^{orb} (\, C'\setminus \mathcal{B}\, , \, y_0 \, ; 
m_1, \dots m_d ) $  as the quotient of $\pi_1(C'\setminus \mathcal{B}, y_0)$
by the minimal normal subgroup generated by the elements $(\gamma_i)^{m_i}$. 
If
$p\colon C \to C'$ is a $G$-covering as above, 
then  its restriction to $C'\setminus\mathcal{B}$ is  a regular
topological cover with short exact homotopy sequence
\[
  1 \ra \pi_1 (\, C\setminus p^{-1}\mathcal{B} \, , \, x_0 \,) \ra \pi_1
(\, C'\setminus \mathcal{B}\, , \, y_0 \, ) \ra G \ra 1.
\]
The corresponding  exact sequence in orbifold covering theory is 

  $$
  1 \ra \pi_1 (\, C \, , \, x_0 \,) \ra \pi_1^{orb}
(\, C'\setminus \mathcal{B}\, , \, y_0 \, ; m_1, \dots m_d) \ra G \ra 1
$$
which is completely determined by the monodromy.  In turn the exact sequence determines,
via conjugation, a homomorphism 
 
$$
\rho \colon G \ra Out^{+} (\pi_1 (\, C \, , \, x_0 
\,)) = Map (C) : = Diff^+ (C) / Diff^0(C)
  $$
 which is fully equivalent to the topological action of $G$ on $C$.
  
  Here the superscript $^+$ denotes orientation-preserving, and the superscript $^0$ denotes `isotopic to the identity'.
The image is contained in the index two subgroup of
outer automorphisms corresponding to orientation preserving
mapping classes under the Dehn-Nielsen-Baer theorem,
cf.\ \cite[thm.\ 8.1]{fm}.

By  Lemma 4.12 of \cite{FabIso}, all the curves $C$ of a fixed 
genus $g$
which admit a given topological action $\rho$ of 
the group $G$ are parametrized by a connected complex manifold; 
arguing as in Theorem 2.4 of \cite{cyclic} we get

\begin{theo}\label{topological action}

The triples $(C,G, \rho)$ where $C$ is a
complex projective curve of genus $g \geq
2$, and $G$ is a finite  group acting effectively on $C$ with a topological action of
type $\rho$
  are  parametrized by a
connected complex manifold  $\sT_{g;G,\rho}$ of
dimension $3 (g'-1) + d $,  where $g'$ is the genus of $C' = C / G$,
and  $d$ is the cardinality of the branch locus $\sB$.

The image  $M_{g,\rho}(G)$ of
  $\sT_{g;G,\rho}$ inside the
moduli space
    $M_g$ is an irreducible closed 
subset of the same dimension  $3 (g'-1) + d $.

\end{theo}

Obviously, composing $\rho$ with an automorphism $\varphi \in Aut(G)$,
i.e. replacing $\rho$ with $\rho \circ \varphi$, does not change the 
 subgroup $\rho (G)\subset Map(C)$. In particular, $M_{g, \rho}(G)=M_{g, \rho\circ \varphi}(G)$,
and similarly $\sT_{g;G,\rho}=\sT_{g;G,\rho\circ \varphi}$.

Notice that $M_g(G)=\bigcup_\rho M_{g, \rho}(G)$, hence the components of $M_g(G)$
are in one-to-one correspondence with a subset of the different  unmarked topological types.
So, the next question which the above result 
motivates is: when do two  monodromies
$\mu_1 , \mu_2 \! : \!  \pi_1
( C' \setminus \sB  ,  y_0  ) \to  G$
 have the same unmarked topological type?

The answer is theoretically easy: the two covering spaces
have the same   unmarked topological type if and only if they are homeomorphic, hence if and only if
$\mu_1$ and $\mu_2$ differ by:
\begin{itemize}
\item
an automorphism of $G$;
\item
and a different choice of a geometric basis. This is 
realized by the action of a mapping class in
$$
Map(g',d):=\frac{Diff^+(C',\mathcal{B})}{Diff^0(C',\mathcal{B})} \, .
$$
\end{itemize}

To reformulate these conditions in terms of Hurwitz generating systems, notice that $Aut(G)$ acts on $HS(G;g',d)$
componentwise, and $Map(g',d)$ acts on $HS(G;g',d)/_{Aut(G)}$. The latter action is given by the 
group homomorphism $Map(g',d)\to Out\left( \pi_1(C'\setminus \mathcal{B},y_0) \right)$ 
and the identification between monodromies $\mu$ and Hurwitz generating systems.  
Theorem \ref{topological action} then implies  that there is  a bijection between  
  the set of unmarked topological   types $[\rho ]$  with $ g' $ and  $d $ fixed, 
and the following orbit space
$$
\{ [ \rho] \} \quad \longleftrightarrow \quad \left( HS(G;g',d)/_{ Aut(G)}\right)/_{Map(g',d)} \, .
$$

In the next sections,  we will also use the action of the unpermuted mapping class  group 
$$
Map^u(g',d+1):=Map^u(C',\mathcal{B}\cup \{ y_0\})
$$ 
on $HS(G;g',d)$, where $Map^u(g',d+1)$
consists of  diffeomorphisms  in $Diff^+ (C')$ which are the identity on $\mathcal{B}\cup \{ y_0\}$, modulo isotopy.
For any 
$v_1, v_2\in HS(G;g',d)$, we write $v_1 \sim v_2$ when  they are  in the same $Map^u(g',d+1)$-orbit.
While, $v_1 \approx v_2$ means that they represent the same class in 
$\left( HS(G;g',d)/_{ Aut(G)}\right)/_{Map(g',d)}$.
Clearly $v_1\sim v_2$ implies $v_1 \approx v_2$.

The mapping class group $Map(g',d)$ acts on $HS(G;g',d)$ only up to conjugation, but, since we are interested 
in classifying Hurwitz generating systems up to $Aut(G)$, we will also  use the notation $\varphi \cdot v$, meaning $\varphi \cdot [v]$,
with $\varphi \in Map(g',d)$ and $[v]\in HS(G;g',d)/_{Aut(G)}$.

\section{The tautological lift}
In this section we give the construction of our invariant in several steps. 
Having defined a suitable  group $G_\Sigma$, for any $\Sigma \subset G$  union of non-trivial 
conjugacy classes, we go on to a map $\e$, which associates to each  Hurwitz vector $v$ (with $c_i \in \Sigma$)
an element $\e (v)\in G_\Sigma$. Any automorphism $f\in Aut(G)$ induces  an isomorphism $f_\Sigma \colon G_\Sigma \to G_{f(\Sigma)}$,
hence $Aut(G)$ acts on the disjoint union $G^\cup=\coprod_\Sigma G_\Sigma$.
We show two key properties of $\e$:
\begin{itemize}
\item
it is $Aut(G)$-equivariant (Lemma \ref{aut-inv}), hence it descends to a map
$$
\tilde{\e}\colon HS(G;g',d)/_{Aut(G)} \to G^\cup/_{Aut(G)} \, ;
$$
\item
$\tilde{\e}$ is constant on the orbits of the mapping class group  $Map(g',d)$ (Proposition \ref{e-inv}).
\end{itemize}
Therefore $\e$ descends to our invariant $\hat{\e}$ which is formalized 
as the map
$$
\hat{\e} \colon  \left( HS(G;g',d)/_{Aut(G)}\right)/_{Map(g',d)} \to G^\cup/_{Aut(G)}\, , \quad  \forall g',d \,  ,
$$
induced by $\e$. We conclude the section with the study of general properties of the invariant 
that are relevant to this work.

Since our construction is inspired by  Hopf's description of the second homology group $H_2(G,\ZZ)$ \cite{Hopf},
we begin by recalling this. For a finite group $G$, fix a  presentation of $G$:
$$
1 \to R \to F \to G \to 1 \, ,
$$
 where $F$ is a free group. Then there is a group isomorphism (cf. \cite{Brown}): 
\begin{equation}\label{Hopf}
H_2(G,\ZZ)\cong \frac{R\cap [F,F]}{[F,R]} \, .
\end{equation}
If $v=(a_1,b_1, \dots , a_{g'},b_{g'})\in G^{2g'}$ satisfies $\prod_1^{g'}[a_i,b_i]=1$, then we can associate  a class in 
$H_2(G,\ZZ)$ in the following way: choose liftings $\widehat{a_i}, \widehat{b_i}\in F$ of $a_i, b_i$, then 
$\prod_1^{g'}[\widehat{a_i}, \widehat{b_i}]\in R\cap [F,F]$ and its class in $\frac{R\cap [F,F]}{[F,R]}$  gives an element of $H_2(G,\ZZ)$,
according to \eqref{Hopf}. Clearly, this element does not  depend on the various choices, moreover it  is invariant 
under the action of the mapping class group,  thus giving a topological invariant of  $v$.

The topological meaning of this invariant is the following. If the $G$-action on $C$ is free, the covering $p\colon C \to C'$
is \'etale, and hence it corresponds to a continuous function $Bp \colon C' \to BG$, up to homotopy. Here $BG$ is the classifying 
space of $G$. The topological invariant is simply the image $Bp_*([C'])\in H_2(BG,\ZZ)=H_2(G,\ZZ)$ of the fundamental class 
$[C']\in H_2(C',\ZZ)$ of $C'$ under the homomorphism $Bp_* \colon H_2(C',\ZZ) \to H_2(BG,\ZZ)$ induced by $Bp$.
Now, if we view $C'$ as an Eilenberg-Mac Lane space $K(\pi_1(C'),1)$, then the fundamental class $[C']$ is given by
$$
\prod_1^{g'}[\widehat{\alpha_i}, \widehat{\beta_i}] \in H_2(\pi_1(C'), \ZZ) \, ,
$$
where as usual $\alpha_1, \beta_1, \dots , \alpha_{g'}, \beta_{g'}$ is a geometric basis of  $ \pi_1(C')$ and $\widehat{\alpha_i}, \widehat{\beta_i}$
are liftings to the  free group of a presentation of  $ \pi_1(C')$. So, 
$$
Bp_*([C'])= \prod_1^{g'}[\widehat{a_i}, \widehat{b_i}] \in H_2(G,\ZZ)\, ,
$$
where $a_i =\mu (\alpha_i)$, $ b_i=\mu (\beta_i)$ and $\mu \colon \pi_1(C')\to G$ is the monodromy of $p\colon C\to C'$.

\medskip

>From now on, $F=\langle \hat{g}\, | \, g\in G\rangle$
is the free group generated by the  elements of $G$. 
Let $R\unlhd F$ be the
minimal normal subgroup generated by the
relations, that is $G=\frac{F}{R}$.

\begin{definition}\label{GGamma}
Let $G$ be a finite group and let $F$, $R$ be as above. 
For any union of non-trivial conjugacy classes $\Sigma \subset G$,  define
\begin{align*}
& R_\Sigma = \langle \langle [F,R], \hat{a}\hat{b}\hat{c}^{-1}\hat{b}^{-1}\, | \, \forall a \in \Sigma , ab=bc\rangle \rangle \, , \\
& G_\Sigma= \frac{F}{R_\Sigma} \, .
\end{align*}
The map $\hat{a}\mapsto a$, $\forall a \in G$, induces a group homomorphism $\a \colon G_\Sigma \to G$.
Set  $K_\Sigma=\ker(\a)$.
\end{definition}

\begin{lemma}\label{centr}
With the notation as before, the following holds. $R_\Sigma \subset R$ and $K_\Sigma=\frac{R}{R_\Sigma}$.
In particular $K_\Sigma$ is contained in the center of $G_\Sigma$ and the short exact sequence
$$
1 \tto \frac{R}{R_\Sigma} \tto G_\Sigma 
\stackrel{\alpha\,}{\tto} G \tto 1\, 
$$
is a central extension.
\end{lemma}
\begin{proof}
 $[F,R]\subset R$ because $R$ is normal in $F$ and moreover   $ \hat{a}\hat{b}\hat{c}^{-1}\hat{b}^{-1}\in R$
for any $a,b,c\in G$ with $ab=bc$, therefore $R_\Sigma \subset R$. By the definition of $\a$ we have that
$K_\Sigma=\frac{R}{R_\Sigma}$. Finally, 
$K_\Sigma$ is in the center of $G_\Sigma$ because  $[F,R]\subset R_\Sigma$.
\end{proof}

The morphism $\a\colon G_\Sigma \to G$ has a \textit{tautological section} $G\to G_\Sigma$, $a\mapsto \hat{a}$.
This map is not a group homomorphism in general, but every element $\xi \in  G_\Sigma$ can be written as
$\hat{g} z = z \hat{g}$, with $g=\alpha(\xi)\in G$ and $z\in K_\Sigma$.
Here, by abuse of notation, $\hat{a}$ denotes also the class of $\hat{a}\in F$ in $G_\Sigma = F/{R_\Sigma}$.

\begin{lemma}
\label{conj-lift}
Let $\hat{a}, \xi \in G_\Sigma$. Suppose that $\hat{a}$  is conjugate to $\xi$ in $G_\Sigma$ and that 
$a\in\Sigma$. Then $\xi=\widehat{\alpha(\xi)}$.\end{lemma}

\proof
Let $\hat{b} z$ be a conjugating element, that is $ \hat{a} \hat{b} z   =  \hat{b} z \xi$. As $z\in K_\Sigma$, it
commutes with any element, hence
\begin{equation}\label{abz}
\hat{a} \hat{b}    =  \hat{b}  \xi \, .
\end{equation}
Now apply $\a$ and obtain: $ab=b\a(\x)$. By assumption $a\in \Sigma$, hence by definition of $G_\Sigma$
we have that $\hat{a}\hat{b}=\hat{b}\widehat{\a(\x)}$. Now using \eqref{abz} we deduce $\x=\widehat{\a(\x)}$.
\qed

\begin{definition}\label{ev}
Given a  $(\gd)$-Hurwitz vector 
$$
v=(c_1, \dots , c_d; a_1, b_1 , \dots , a_{g'}, b_{g'})
$$
 in $G$, 
 cf.\ Definition \ref{factorisations}, 
its \emph{tautological lift} $\hat{v}$ 
is the  $(\gd)$-Hurwitz vector in $G_\Sigma$ defined by
 $$
 \hat{v}=(\widehat{c_1}, \dots , \widehat{c_d}; \widehat{a_1}, \widehat{b_1}, \dots ,  \widehat{a_{g'}}, \widehat{b_{g'}})
 $$ 
 where the factors are the tautological lifts of the factors of $v$. 
 
Given a  $(\gd)$-Hurwitz vector $v$ in $G$ with $c_i\not= 1$, 
$\forall i$, we denote by $\Sigma_v$ the union of all conjugacy classes of $G$
containing at least one $c_i$.

For any $v$ as before,  let
$$
\e(v)=\prod_1^d \widehat{c_j} \cdot \prod_1^{g'}[\widehat{a_i}\widehat{b_i}]\in G_{\Sigma_v} \, ,
$$ 
be
the evaluation of the tautological lift $\hat{v}$ of $v$ in $G_{\Sigma_v}$, 
in analogy to Definition \ref{factorisations}.
\end{definition}

\begin{lemma}\label{aut-inv}
Let $G$ be any finite group, and let $\Sigma \subset G$ be any union of non trivial 
conjugacy classes. Then we have:
\begin{itemize}
\item[i)]
any $f\in Aut(G)$ induces an isomorphism
$f_\Sigma \colon G_\Sigma \to G_{f(\Sigma)}$;
\item[ii)]
$\e(f(v))=f_\Sigma (\e(v))$, $\forall f\in Aut(G)$ and $\forall v$ a $(\gd)$-Hurwitz vector with $c_i\not= 1$, $\forall i$, where $\Sigma =\Sigma_v$.
\end{itemize}
\end{lemma}
\begin{proof}
i)  $f\in Aut(G)$ lifts to an automorphism $\hat{f}\in Aut(F)$ defined by
$$
\hat{f} : \hat{a} \:\mapsto\: \widehat{{f(a)}} \, .
$$
We have: $\hat{f} (R) \subset R$, and moreover
$$
\hat{f}(\hat{a}\hat{b}\hat{c}^{-1}\hat{b}^{-1}) \quad = \quad 
\widehat{f( a)} \widehat{f( b)} \widehat{f( c)}^{-1} \widehat{f( b)}^{-1} \, ,
$$
for any  $a,b,c\in G$. If $a\in \Sigma$, then $f(a)\in f(\Sigma)$ and hence 
$$
\widehat{f( a)} \widehat{f( b)} \widehat{f( c)}^{-1} \widehat{f( b)}^{-1}\in R_{f(\Sigma)} \, .\medskip
$$
\medskip\medskip
ii) \vspace{-\baselineskip}\vspace{-\baselineskip}
\begin{align*}
\e(f(v))&= \e(f(c_1), \dots , f(c_d); f(a_1) , \dots , f(b_{g'})) \\
&= \prod_1^d\widehat{f(c_i)}\cdot \prod_1^{g'}[\widehat{f(a_j)},\widehat{f(b_j)} ]\\
&= \prod_1^d\hat{f}(\widehat{c_i})\cdot \prod_1^{g'}[\hat{f}(\widehat{a_j}),\hat{f}(\widehat{b_j}) ]=f_\Sigma (\e(v))\, .
\end{align*}
\end{proof}

Now, we define 
$$
G^{\cup} :  = {\coprod}_\Sigma G_\Sigma \, , 
$$
and regard $\e$ as a map  $\e\colon HS(G;g',d)\to G^\cup$, $v\mapsto \e(v)\in G_{\Sigma_v}$.
Then the previous lemma means that $\e$ induces a map
$$
\tilde{\e}\colon HS(G;g',d)/_{Aut(G)} \to \left( G^\cup \right)/_{Aut(G)} \, .
$$
We have the following
\begin{prop}\label{e-inv}
For any $g', d\in \NN$,  $\tilde{\e}$ is  $Map(g',d)$-invariant, hence it induces a map
$$
\hat{\e}\colon \left( HS(G;g',d)/_{Aut(G)}\right)/_{Map(g',d)} \to \left( G^\cup \right)/_{Aut(G)} \,  .
$$
\end{prop}
To prove this proposition we need some preliminary results. 
\begin{lemma}\label{braid-invar}
Let $\Sigma_v$ be associated to
a $(\gd)$-Hurwitz vector $v$ as in Definition \ref{ev}.
If the Hurwitz vector
$v'$ is related to $v$ by an elementary braid move,
then $\e(v)=\e(v')$.
\end{lemma}
\begin{proof}
It suffices to consider the case $g=0$, $d=2$ and the
elementary braid
move $\sigma_1$ which maps 
$$
v = (c_1, c_2)\quad\text to\quad
v' = ( c_2, c_2^{-1}  c_1c_2 ) \, .
$$
In $G_{\Sigma_v}$ we have, thanks to $c_1\in \Sigma_v$, $c_1c_2= c_2(c_2^{-1} c_1 c_2)$,
and the relations of $G_{\Sigma_v}$:
\begin{align*}
\e(v)  = \widehat{c_1} \widehat{c_2}  =  \widehat{c_2} \widehat{c_2^{-1}c_1 c_2}  =  \e({v'}) \, .
\end{align*}
\end{proof}

\begin{lemma}
\label{abel}
If $\x,\eta \in G_\Sigma$, then
$$
\left[ \x, \eta \right] \quad = \quad \left[\widehat{\alpha(\x)}, \widehat{\alpha(\eta)}\right] \, .
$$
\end{lemma}
\begin{proof}
Write $\x=\widehat{\alpha(\x)} z$ and $\eta=\widehat{\alpha(\eta)}z'$
with $z,z'$ in $K_\Sigma$, hence central (Lemma \ref{centr}). 
Then the conclusion is immediate.

\end{proof}

\begin{proof} (Of Proposition \ref{e-inv}.) Let $\varphi \in Map(g',d)$.
Thanks to Lemma \ref{braid-invar} it suffices to consider
the case that $\varphi$ is a pure mapping class, i.e.
that $\varphi$ does not permute the conjugacy classes
associated to the local monodromies.
Using Lemma \ref{aut-inv} ii), we can further ignore the action of $G$ by conjugation
and pretend that $Map(g',d)$ acts on $HS(G;g',d)$.

Since $\varphi$ is a pure mapping class, 
the components of $v$, $\varphi\cdot v$ are conjugate to each
other, $v_i\sim (\varphi \cdot v)_i$  for $i=1,\ldots , d$, and where   $\sim$ denotes  conjugation equivalence. For the
components of $\hat{v}, \varphi\cdot\hat{v}$ the same is true,
$\hat{v}_i\sim (\varphi \cdot \widehat{v})_i$  for $i=1,\dots,d$.
By Lemma \ref{conj-lift}
we thus have:
$$
(\hat{v})_i  \sim  (\varphi \cdot \hat{v})_i \, \Rightarrow  (\varphi \cdot \hat{v})_i  = 
\widehat{\alpha(\!(\!\varphi\! \cdot\! \hat{v}\!)_{\!_{i}}\!\!\:) } \, .
$$

Now notice that the homomorphism $\alpha$ of Definition
\ref{GGamma} induces a map
\[
\alpha^{(d+2g')}: HS(G_\Sigma; g',d) \quad \tto \quad
HS(G; g',d),
\]
which is equivariant under the action of the mapping class group in the following sense:
consider the factorizations as a map from the free group on
$d+2g'$ generators to $G_\Sigma$, resp. $G$, and
the mapping class group as a group of automorphisms of this
free group.
Then $\alpha^{(d+2g')}$ is equivariant, since such automorphisms
act by pre-composition. Hence by the equivariance of $\alpha^{(d+2g')}$
$$
\alpha((\varphi \cdot \hat{v})_i) \quad = \quad (\varphi \cdot v)_i
$$
and therefore, for $i=1,\dots, d$,
$$
(\varphi \cdot \hat{v})_i  =
\widehat{\alpha(\!(\!\varphi\! \cdot\! \hat{v}\!)_{\!_{i}}\!\!\:)}
=  \widehat{(\varphi \cdot v)_i}  =  (\widehat{\varphi \cdot v})_i \, .
$$

By Lemma \ref{abel} we may change also the entries
$(\varphi \cdot \hat{v})_i$, $i>d$ in the commutators to
$\widehat{\alpha(\!(\!\varphi\! \cdot\! \hat{v}\!)_{\!_{i}}\!\!\:) }=(\widehat{\varphi \cdot v})_i$ without changing the
value of the commutators. Hence
$$
ev( \varphi \cdot \hat{v} ) \quad = \quad ev ( \widehat{\varphi \cdot v} ) \quad = \quad \e (\varphi \cdot v)\, .
$$

By the invariance of the evaluation under the mapping class
$$
\e(v) \quad = \quad ev(\hat{v}) \quad = \quad
ev(\varphi \cdot \hat{v}) 
 \quad = \quad \e(\varphi\cdot v)
$$
and we have proved our claim.
\end{proof}

\begin{definition}\label{Nielsen}
Let $v\in HS(G;g',d)$ and let $\nu(v) \in \oplus_{\mathcal{C}}\ZZ\langle \mathcal{C}\rangle$ ($\mathcal{C}$ runs over the set of conjugacy classes of $G$)
 be the vector whose $\mathcal{C}$-component
is the number of $v_j$, $j\leq d$, which belong to $\mathcal{C}$
 ($\nu(v)$ is also called the shape of $v$ in \cite{FV}).\\
The map
$$
\nu \colon HS(G;g',d) \to \oplus_{\mathcal{C}}\ZZ\langle \mathcal{C}\rangle 
$$
obtained in this way  induces a map 
$$
\tilde{\nu}\colon HS(G;g',d)/_{Aut(G)} \to \left( \oplus_{\mathcal{C}} \ZZ \langle  \mathcal{C} \rangle \right)/_{Aut(G)}
$$
which is $Map(g',d)$-invariant, therefore we get  a map 
$$
\hat{\nu}\colon \left( HS(G;g',d)/_{ Aut(G)}\right)/_{Map(g',d)} \to ( \oplus_{\mathcal{C}}\ZZ\langle \mathcal{C}\rangle)/_{Aut(G)} \, .
$$
For any $v\in HS(G;g',d)$, we call $\hat{\nu}(v)$ the $\nu$-type of $v$.
\end{definition}

\begin{rem}\label{evsnu}
Let $v\in HS(G;g',d)$  and let $\Sigma_v\subset G$ be the union of the conjugacy classes of $v_j$, $j\leq d$.
The abelianization $G_{\Sigma_v}^{ab}$ of $G_{\Sigma_v}$ can be described as follows:
$$
G_{\Sigma_v}^{ab}\cong \oplus_{\mathcal{C}\subset \Sigma_v}\ZZ\langle \mathcal{C} \rangle \oplus \oplus_{g\in G\setminus \Sigma_v}\ZZ \langle g \rangle \, ,
$$
where $\mathcal{C}$ denotes a conjugacy class of $G$.

Observe that $\nu(v) \in \oplus_{\mathcal{C}\subset \Sigma_v}\ZZ\langle \mathcal{C} \rangle \subset G_{\Sigma_v}^{ab}$
 coincides with the vector which is  the image in $G_{\Sigma_v}^{ab}$
of $\e(v)\in G_{\Sigma_v}$
under the natural homomorphism  $G_{\Sigma_v} \to G_{\Sigma_v}^{ab}$. \end{rem}

\begin{definition}\label{H2Gamma}
Let $\Sigma \subset G$ be a union of non-trivial conjugacy classes of $G$. We define
$$
H_{2,\Sigma}(G)=\ker \left( G_\Sigma \to G\times G_\Sigma^{ab} \right) \, ,
$$
where $G_\Sigma \to G\times G_\Sigma^{ab}$ is the morphism with first component $\a$ (defined in Definition \ref{GGamma})
and second component the natural morphism $G_\Sigma \to  G_\Sigma^{ab}$.
\end{definition}
Notice that 
$$
H_2(G,\ZZ)\cong \frac{R\cap [F,F]}{[F,R]} \cong \ker \left( \frac{F}{[F,R]} \to G\times G^{ab}_\emptyset \right) \, .
$$
In particular, when $\Sigma = \emptyset$, $H_{2,\Sigma}(G)\cong H_2(G, \ZZ)$.

The next result gives a precise relation between $H_2(G,\ZZ)$ and $H_{2,\Sigma}(G)$.
\begin{lemma}\label{H2toH2Gamma}
Let  $G$ be a finite group and let  $\Sigma \subset G$ be a union of nontrivial conjugacy classes.
Write $G=\frac{F}{R}$ and  $G_\Sigma =\frac{F}{R_\Sigma}$.
Then,  there is a short exact sequence
$$
1\to \frac{R_\Sigma \cap [F,F]}{[F,R]}\to H_2(G,\ZZ)\to H_{2,\Sigma}(G)\to 1 \, .
$$
In particular $H_{2,\Sigma}(G)$ is abelian.
\end{lemma}
\begin{proof}
We first define the morphism $H_2(G,\ZZ)\to H_{2,\Sigma}(G)$.\\
By  Hopf's Theorem  we identify $H_2(G,\ZZ)$ with $\frac{R\cap [F,F]}{[F,R]}$ (cf. \cite{Brown}). 
On the other hand we have: 
\begin{align*}
H_{2,\Sigma}(G)= \ker \left( G_\Sigma \to G\right) \cap \ker \left( G_\Sigma \to G_\Sigma^{ab}\right)  = \frac{R}{ R_\Sigma} \cap[G_\Sigma , G_\Sigma]\, .
\end{align*}
By Lemma \ref{centr},  $R_\Sigma \subset R$.
The homomorphism $R\cap [F,F]\to \frac{R}{ R_\Sigma}$, $r\mapsto rR_\Sigma$, takes values in $H_{2,\Sigma}(G)$.
Moreover it descends to a group homomorphism $H_2(G,\ZZ)\to H_{2,\Sigma}(G)$ because $[F,R]\subset R_\Sigma$.

To prove the surjectivity, let 
$$
aR_\Sigma \in \frac{R}{R_\Sigma}\cap [G_\Sigma,G_\Sigma] \, .
$$
Since $aR_\Sigma \in [G_\Sigma , G_\Sigma] = \frac{[F,F]\cdot R_\Sigma}{R_\Sigma}$, we may assume $a\in [F,F]$.
>From $aR_\Sigma \in \frac{R}{R_\Sigma}$, we have $aR_\Sigma=rR_\Sigma$, for some $r\in R$.
Since $R_\Sigma \subset R$, we deduce that $a\in R$ and so the surjectivity follows.

The kernel of the morphism so defined is $\frac{R_\Sigma \cap[F,F]}{[F,R]}$.

Since $H_2(G,\ZZ)$ is abelian, so is $H_{2,\Sigma}(G)$.
\end{proof}

\begin{prop}
Let $v_1 , v_2\in HS(G;g',d)$ be two Hurwitz generating systems in $G$ with the same $\nu$-type.
Then $\Sigma_{v_1}=\Sigma_{v_2}=:\Sigma$. Moreover, if $ev(v_1)=ev(v_2)\in G$, then the element
$$
\e({v_1})^{-1} \cdot \e({v_2}) \:\in\: 
H_{2,\Sigma}(G)
$$
is invariant under  the  group $\widetilde{Map}(g',d)$ of isotopy classes of orientation-preserving diffeomorphisms
of the pair $(C',\mathcal{B})$ that fix $y_0$. In particular:
\begin{enumerate}
\item
if $v_1$and $ v_2$ are equivalent then the element is trivial;
\item
if the element is non-trivial, then $v_1$ and $v_2$  are 
inequivalent.
\end{enumerate} 
\end{prop}


\section{Computation of $H_{2,\Sigma}(D_n)$}

In this section we derive a complete description of 
$H_{2,\Sigma}(G)$ in the special case that $G$ is equal
to the dihedral group 
\begin{eqnarray*}
D_n & = & \langle\: x,y \:|\: x^n = 1,\: y^2 = 1,\: xy = yx^{-1} \:\rangle \\
& = & \left\{\: x^iy^j \:|\: 0\leq i < n,\: 0 \leq j < 2 \:\right\}
\end{eqnarray*}

\begin{prop}\label{H2Dn}
Let $n\in \NN$, $n\geq 3$. Then we have:
\begin{enumerate}
\item[(i)] $H_2(D_n , \mathbb{Z})$ is trivial if $n$ is odd and it is isomorphic to $\mathbb{Z}/2\mathbb{Z}$
if $n$ is even;
\item[(ii)] the natural action of ${Aut}(D_n)$ on $H_2(D_n , \mathbb{Z})$ is trivial.
\end{enumerate}
\end{prop} 

\proof
(ii) This claim follows directly from (i) and from the fact that the neutral element of $H_2(D_n , \mathbb{Z})$ is fixed by the action of ${Aut}(D_n)$.
 
(i) Identify $D_n$ with the subgroup of $SO(3)$ generated by 
$$
x :=\left( \begin{matrix} \cos \frac{2\pi}{n} & -\sin \frac{2\pi}{n} & 0 \\ \sin \frac{2\pi}{n} & \cos \frac{2\pi}{n} & 0 \\ 0&0&1\end{matrix}\right) \quad
\mbox{and}\quad y := \left( \begin{matrix} -1 & 0 & 0 \\ 0&1&0\\ 0&0&-1 \end{matrix}\right) \, .
$$
Let $u \colon SU(2) \ra SO(3)$ be the homomorphism $q\mapsto R_q$, where we identify $SU(2)$ with the quaternions $q\in \HH$
of norm $1$, $\RR^3$ with $Im\HH$, and $R_q(x)=qx\ol{q}$.
Consider the binary dihedral group  $\tilde{ D}_n = u^{-1}(D_n)$. It fits in the following short exact sequence:
\begin{equation}\label{binary Dn}
1\ra \ZZ/2\ZZ \ra \tilde{ D}_n \ra D_n \ra 1\, ,
\end{equation}
from which we get the   $5$-term exact sequence (see e.g. \cite{Brown}, pg. 47, Exercise 6):
\begin{equation}\label{5-term}
H_2(\tilde{ D}_n) \ra H_2(D_n ) \ra \left( H_1(\ZZ/2\ZZ)\right)_{D_n} 
\ra H_1(\tilde{ D}_n) \ra H_1(D_n)\ra 0\, , 
\end{equation}
where all the coefficients are in $\ZZ$ and $\left( H_1(\ZZ/2\ZZ)\right)_{D_n}$ is the group of co-invariants under  the $D_n$-action
 on $\ZZ/2\ZZ$ induced by conjugation 
by $\tilde{D}_n$, hence $\left( H_1(\ZZ/2\ZZ)\right)_{D_n}= H_1(\ZZ/2\ZZ)$ since $\ZZ/2\ZZ$ is in the center of $\tilde{\rm D}_n$.\\
We have that  $H_2(\tilde{D}_n)=\{ 0\}$, since $\tilde{D}_n$ is a finite subgroup
of $SU(2)\cong S^3$ (see \cite{Brown} II pg. 47 , Exercise 7). 
Next, recall that, for any group $G$, $H_1(G,\ZZ)$ is isomorphic to the abelianization $G^{ab}$ (see \cite{Brown} pg. 36), hence
\eqref{5-term} reduces to 
$$
0 \ra H_2(D_n ) \ra \ZZ/2\ZZ \ra \tilde{D}^{ab}_n \ra \Dn^{ab}\ra 0 \, .
$$
To conclude we show that  ${ \ker} ( \tilde{ D}^{ab}_n \ra D_n^{ab} ) =\{ 0 \}$ if and only if $n$ is even. 
With the imaginary units $\underline{i},\underline{j},\underline{k}
\in \mathbb H$ let
 $$
 {\xi}= \cos \left( \frac{\pi}{n} \right) 
 +\underline{k}\cdot \sin \left( \frac{\pi}{n} \right) \in u^{-1}(x) 
 \quad \mbox{and} \quad 
 {\eta}= \underline{j}\in u^{-1}(y)\, .
 $$ 
 Since $[ {\xi}^\ell ,  {\eta}]={\xi}^{2\ell}$, $\forall \ell$, we see that, if $n$ is odd, $\xi^n \not\in [\tilde{D}_n, \tilde{D}_n]$, but $u(\xi^n)=1$
 and hence ${ \ker} ( \tilde{ D}^{ab}_n \ra D_n^{ab} ) \not=\{ 0 \}$. When $n$ is even, $\tilde{D}_n^{ab}\cong \ZZ/2\ZZ\times \ZZ/2\ZZ$
 and the map $ \tilde{ D}^{ab}_n \ra D_n^{ab} $ is an isomorphism.
   \qed
 
\medskip
 
Using Lemma 3.2 from \cite{Wiegold}, we deduce the following
\begin{cor}\label{Schur cover}
Let $n\in \NN$, $n\geq 4$ even. Then, the binary dihedral group  $\tilde{ D}_n$ is a Schur cover of $D_n$
and the exact sequence \eqref{binary Dn} identifies $\ZZ/2\ZZ$ with $H_2(D_n,\ZZ)$.
In particular, for any $(a_1,b_1, \dots , a_{g'}, b_{g'}) \in (D_n)^{2g'}$ with $\prod_1^{g'}[a_i,b_i]=1$, the image 
of $\prod_1^{g'}[\widehat{a_i},\widehat{b_i}]\in R\cap [F,F]$ in $H_2(D_n,\ZZ)=\frac{R\cap [F,F]}{[R,F]}$ is given
by $\prod_1^{g'}[\tilde{a}_i, \tilde{b}_i]$, where $\tilde{a}_i, \tilde{b}_i\in \tilde{D}_n$ are liftings of $a_i, b_i$. 
\end{cor}

Next the union $\Sigma$ of non-trivial conjugacy classes comes
into play. Recall that the set of reflections
$\{ x^iy \;|\; 0\leq i<n \}$
is a single conjugacy class in case $n$ is odd, and splits
into two conjugacy classes in case $n$ even according to the
parity of $i$.

The conjugacy classes in the set of rotations
$\{ x^i \;|\; 0\leq i<n \}$
are of the kind $\{ x^i, x^{n-i}\}$ and contain two elements
except for $i=0$ and $i=\frac n2$ in case $n$ is even.

\begin{cor}\label{H2GammaD}
Let $\Sigma \subset D_n$ be  a
 union of non-trivial conjugacy classes, $\Sigma \not= \emptyset$. Then $H_{2,\Sigma}(D_n)=\{0\}$ in the following cases: 
$n$ is odd;
$n$ is even and $\Sigma $  contains some reflection; $n$ is even and  $\Sigma $  contains 
the non-trivial central element. In the remaining case, $H_{2,\Sigma}(D_n)=\ZZ/2\ZZ$.
\end{cor}
\begin{proof}
 If $n$ is odd, then $H_2(D_n,\ZZ)=\{0\}$ and hence  $H_{2,\Sigma}(D_n)=\{0\}$ for any $ \Sigma$
 (Lemma \ref{H2toH2Gamma}).
 
If $n=2k$ and $\Sigma$ contains some reflection, say $y\in \Sigma$, then $\hat{y}\widehat{x^k}\hat{y}^{-1}\widehat{x^k}^{-1}\in R_\Sigma \cap [F,F]$.
But the image of this element in $H_2(D_n,\ZZ)$ is not trivial (Corollary \ref{Schur cover}), hence  $H_{2,\Sigma}(D_n)=\{0\}$ (Lemma \ref{H2toH2Gamma}).
The same argument works if $xy \in \Sigma$. 

Assume now $n=2k$ and $x^k\in \Sigma$. Then $\widehat{x^k}\hat{y}\widehat{x^k}^{-1}\hat{y}^{-1}\in R_\Sigma \cap [F,F]$
and its image in $H_2(D_n,\ZZ)$ is not trivial, hence $H_{2,\Sigma}(D_n)=\{ 0\}$ also in this case.

Finally, if $n=2k$ and $\Sigma \subset \ZZ/n\ZZ \setminus \{ x^k \}$, then 
$$
R_\Sigma =\langle \langle [F,R], \, \widehat{x^\a} \widehat{x^\b} \widehat{x^\a}^{-1}\widehat{x^\b}^{-1}, \, 
\widehat{x^\a} \widehat{x^\b y} \widehat{x^{n-\a}}^{-1}\widehat{x^\b y}^{-1} \, | \, x^\a \in \Sigma\rangle \rangle\, .
$$
First we note that
the image of $ \widehat{x^\a} \widehat{x^\b} \widehat{x^\a}^{-1}\widehat{x^\b}^{-1}$ in $H_2(D_n,\ZZ)$ is $0$.
Second, the elements $\widehat{x^\a} \widehat{x^\b y} 
\widehat{x^{n-\a}}^{-1}\widehat{x^\b y}^{-1}$ generate an abelian
group modulo $[F,R]$.
Last, the intersection of this subgroup with $[F,F]/[F,R]$ is generated
by elements represented by 
\begin{eqnarray*}
&&
\widehat{x^\a} \widehat{x^\b y} \widehat{x^{n-\a}}^{-1}\widehat{x^\b y}^{-1}
\quad\cdot\quad
\widehat{x^{n-\a}} \widehat{x^\gamma y} \widehat{x^{\a}}^{-1}\widehat{x^\gamma y}^{-1} \, .
\end{eqnarray*}
It remains to show that these are trivial modulo $[F,R]$, in fact
\begin{eqnarray*}
& \equiv & \widehat{x^\b y}^{-1}
\widehat{x^\a} \widehat{x^\b y} \widehat{x^{n-\a}}^{-1}
\quad\cdot\quad
\widehat{x^{n-\a}} \widehat{x^\gamma y} \widehat{x^{\a}}^{-1}\widehat{x^\gamma y}^{-1} \\
& \equiv & \widehat{x^\b y}^{-1}
\widehat{x^\a} \widehat{x^\b y} 
\quad
 \widehat{x^\gamma y} \widehat{x^{\a}}^{-1}\widehat{x^\gamma y}^{-1}\\
& \equiv & \widehat{x^\b y}^{-1}
\widehat{x^\a} \quad 
\underbrace{
\widehat{x^\b y} \quad 
 \widehat{x^\gamma y} 
 \quad \widehat{x^{\gamma-\b}} }_{\in R}
 \quad \widehat{x^{\gamma-\b}}^{-1} \quad
 \widehat{x^{\a}}^{-1}\widehat{x^\gamma y}^{-1}\\
& \equiv & \widehat{x^\b y}^{-1}
\quad \widehat{x^\b y} \quad 
 \widehat{x^\gamma y} 
 \quad \widehat{x^{\gamma-\b}} \quad
\widehat{x^\a}  \quad \widehat{x^{\gamma-\b}}^{-1} \quad
 \widehat{x^{\a}}^{-1}\widehat{x^\gamma y}^{-1} \\
& \equiv &  \widehat{x^\gamma y} 
 \quad \widehat{x^{\gamma-\b}} \quad
\widehat{x^\a}  \quad \widehat{x^{\gamma-\b}}^{-1} \quad
 \widehat{x^{\a}}^{-1}\widehat{x^\gamma y}^{-1} \\
& \equiv &   \widehat{x^{\gamma-\b}} \quad
\widehat{x^\a}  \quad \widehat{x^{\gamma-\b}}^{-1} 
 \widehat{x^{\a}}^{-1} \, .
\end{eqnarray*}
This last element is trivial modulo $[F,R]$ as noted first. We deduce
 that $\frac{R_\Sigma \cap [F,F]}{[F,R]}=\{ 0 \}$ and hence $H_{2,\Sigma}(D_n)\cong H_2(D_n,\ZZ)\cong \ZZ/2\ZZ$, by Lemma \ref{H2toH2Gamma}.

\end{proof}


\section{The injectivity of $\hat{\e}$ when $G=D_n$}
Recall the following

\noindent \textbf{Notation.} For any Hurwitz vector  $v=(c_1,\, \dots \, , c_d; a_1, b_1, \, \dots \, , a_{g'}, b_{g'})\in G^{d+2g'}$, 
$$
ev(v)=\prod_{i=1}^d c_i\cdot \prod_{j=1}^{g'}[a_j,b_j] \in G \, ,
$$
while, if $c_i\not=1$, $\forall i$, $\e (v)=ev (\hat{v})\in G_{\Sigma_v}$, where $\hat{v}\in  (G_{\Sigma_v})^{d+2g'}$ is the tautological lifting (Definition \ref{ev}).

\medskip

In this section we prove the following
\begin{theo}\label{Dn-case}
Let $G=D_n$, the dihedral group of order $2n$. Then, $\forall g', d$, we have:
\begin{itemize}
\item[(i)]
$\hat{\e}\colon \left( HS(G;g',d)/_{Aut(G)}\right)/_{Map(g',d)} \to  (G^\cup)/_{Aut(G)} : =( {\coprod}_\Sigma G_\Sigma)/_{Aut(G)}  $

 is injective;
\item[(ii)]
The image $Im(\hat{\e})$ is the inverse image of $Im(\hat{\nu})$  in $( {\coprod}_\Sigma K_\Sigma)/_{Aut(G)}$,
where  $K_\Gamma$ is defined in Def. \ref{GGamma}. 
In other words, for any $\Gamma \subset G$, union of non trivial conjugacy classes, 
$Im({\e})\cap G_\Gamma = H_{2,\Gamma}(G)\cdot (Im(\e)\cap G_\Gamma)$.
\item[(iii)] For $g' \geq 2$,  $Im(\hat{\nu})$ is just the set of admissible class functions $\nu$. For $g' = 1$, 
$Im(\hat{\nu})$ is the union of
the set of admissible class functions $\nu$   for which $\Ga$ contains  some reflection, together with a  subset $\sS$ of those  for which $\Ga$ generates a subgroup 
$H$ of index at most $2$ in the subgroup $\sR$ of rotations, $\sS$  contains the subset for which $H = \sR$.

\end{itemize}
\end{theo}

To prove (i), let $[v_1]_\approx , [v_2]_\approx \in (HS)/_\approx$
be equivalence classes up to the combined action of $Aut(G)$
and the mapping class group
such that
$\hat{\e}([v_1]_\approx)=\hat{\e}([v_2]_\approx)$.
Then there exists an automorphism $f \in Aut(G)$ such that $f (\Sigma_{v_1})=\Sigma_{v_2}$ and $f (\e(v_1))=\e (v_2)$.
Hence, by Lemma \ref{aut-inv},  we assume without loss of generality $\Sigma_{v_1}=\Sigma_{v_2}=\Sigma$ and $\e(v_1)=\e(v_2)$, in particular
$$
\e (v_1)\cdot \e(v_2)^{-1}=0 \in H_{2,\Sigma}(G)\, .
$$

The outline of the proof is now the following. We address the following mutually 
exclusive cases:
$\Sigma=\emptyset$ (the \'etale case); $\Sigma \not=\emptyset$ and contains some reflection; 
$\Sigma \not=\emptyset$ and does not contain reflections.
In the first case, for each element  of $ HS$, we determine a normal form with respect to $\approx$,
then we show that two different normal forms are distinguished by $H_{2}(D_n,\ZZ)$  
(recall that $Aut(D_n)$ acts trivially on $H_{2}(D_n,\ZZ)$).
In the second case we will show that all  Hurwitz generating systems with the same  numerical invariants 
($n$, $g'$ and $\nu$-type)
are equivalent with respect to $\approx$ (this agrees with the fact $H_{2,\Sigma}(D_n)=\{ 0 \}$ in this case).  
In the last case, for every $v\in HS$, we determine a normal form $v'$ with respect to $\approx$.
We see that two different normal forms $v_1'$ and $v_2'$ have different invariants, 
$\e({v_1'})\not= \e({v_2'}) \in G_\Sigma$.
Finally we prove that $v_1\approx v_2$ if and only if $\exists f \in Aut(D_n)$ such that 
$f(\Sigma)=\Sigma$ and $f(\e({v_1'}))= \e({v_2'})$.
>From this  (i) follows. We refer to \cite{CLP11} for a useful description of $Aut(D_n)$.

To prove (ii), we observe that for any $v\in HS(D_n;g',d)$ the orbit $\epsilon(v)\cdot H_{2,\Gamma_v}(D_n)$
is either $\{ \e(v) \}$ or $\{ \e(v) , -\e(v) \}$ (cf. Cor. \ref{H2GammaD}). In the proof of (i) we list all possible normal forms for 
Hurwitz generating systems and we will see that, 
in the case where $H_{2,\Gamma_v}(D_n)\cong \mathbb{Z}/2\mathbb{Z}=\{ \pm 1\}$,
there exists $v'\in HS(D_n;g',d)$ with $\Gamma_{v'}=\Gamma_v$ and $\e(v')\cdot \e(v)^{-1}=-1$.

To prove (iii), observe that, given an admissible class function $\nu$, and elements $c_1, \dots ,c_d$ which
yield the given function $\nu$, the product  $c_1\cdot  \dots \cdot  c_d : = c$ is in the commutator subgroup. However,
in the dihedral group the commutator subgroup is equal to the set of commutators. Hence we may find $a_1, b_1$ such 
that $ c^{-1}  = [a_1, b_1]$, it suffices, if $ c = x^{2\d}$, to take $a_1 = x^my, b_1 =  x^{m+\d}y$.

If some $c_j$ is a reflection, without loss of generality, we may assume that $c_j = y$, hence, choosing $m=1$,
we obtain that 
$$v : = (c_1, \dots c_d; xy,  x^{1+\d}y,1,  \dots ,1)$$  is a Hurwitz generating system.
If $g' \geq 2$, we can just take the Hurwitz generating system
$$v : = (c_1, \dots c_d; y,  x^\d y ,x, 1,  \dots ,1).$$

In the case where $g' = 1$ and all the $c_j$ are rotations, observe that 
$$v : = (c_1, \dots c_d; x^my,  x^{m+\d}y)$$
is a Hurwitz generating system if the dihedral group is generated by $H$ (the subgroup generated by the $c_j$ 's),
together with $x^my$ and $x^\d$. This amounts to the condition that $H$ and $x^\d$ generate the subgroup $\sR$ of rotations.
Since $ c = x^{2\d} \in H$, we see that a necessary condition is that $H$ has index at most two in the group $\sR$ of rotations,
and a sufficient one is that $H = \sR$. In the case where the index is exactly two, so that $n = 2h$ is even, the condition is that 
$\d$ is an odd number. Observe that in this case if we  replace some $c_j = x^{2i}$ by its inverse, the number $\d$
is replaced by $ \d - 2i$, so that condition is indeed a condition on the class function $\nu$. 

\begin{rem}
For the case $g'=0$ we  defer to our previous article \cite{CLP11}, where all the normal forms were given.

But we can give a direct description  as follows.

 For $g'=0$, $Im(\hat{\nu})$ is just the set of admissible class functions $\nu$ satisfying one of  the following conditions:

(R) there are only two reflection classes, and the subgroup generated by the rotation classes
is  the whole $\sR$

(O)  $n$ is  odd and  the class function $\nu$   takes  value at least $4$ on
the reflection class 

(E) $n$ is  even,  there are at least four reflection classes and   either the class function $\nu$  takes positive value on
both  reflection classes,  or there exists a rotation class with odd exponent. 

In fact, since we want a Hurwitz generating system, it is a necessary condition that there must  be at least one 
reflection (and indeed an even number  of reflections
by admissibility). Take now an admissible class satisfying this restriction, and
take  $c_1, \dots, c_d$   realizing the admissible class.
Then their product $c$ is in the commutator subgroup, so it is any rotation
in case $n$  is odd and a rotation of type $x^{2i}$ in case $n$ is even.

Since there is a  reflection, we replace the last reflection, say $c_i$, by
$c_i c^{-1}$, and obtain thus an admissible Hurwitz  
system without changing the class function. 

We must however have an admissible  Hurwitz generating system, 
and we recall that there is   an even number of reflections among the $c_i$'s.

Up to automorphisms of $D_n$, we have reflections $y$ and $ x^{m}y$ in the components of the Hurwitz vector.

 If  we have only two reflections $y, x^{m}y$,  we are done if and only  if the subgroup generated by the rotations is 
 the whole subgroup $\sR$, since $ x^{-m} \in \sR$.

 Assume that  $n$ is odd and there are 4 reflections: assume that the first four are  $y,   x^{m_1}y, x^{m_2}y, x^{m_3}y$:
  then we can replace these
 (without changing the class function)  by 
 the following reflections $y, xy,  x^{m_2- m_1 + 1}y,  x^{m_3}y$ and we have obtained  a Hurwitz generating system.
 
 Assume that $n$ is even and there are  at least four reflections:  
  $$y,   x^{m_1}y,   x^{m_2}y,  x^{m_3}y, \dots .$$ 
  We can change the rotations by adding to $m_i$ an even integer $2h_i$, in such a way that the 
  sum $ -h_1 + h_2 - h_3 \dots \equiv 0 $. Hence if some $m_i$ is odd, we obtain $y$ and $xy$, and we have a generating system.
  If instead all $m_i$'s  are even, we obtain $y$ and $x^2y$, 
  hence we are finally done if and only if there is a rotation with odd exponent.

\end{rem}

\medskip

\noindent \textbf{Case 1: $\Sigma=\emptyset$ (the \'etale case).} 
In this case $H_{2,\Sigma}(D_n)=H_2(D_n,\ZZ)$, so $v\in HS(D_n)$ implies  $\e (v)\in H_2(D_n,\ZZ)$.
In the following, we identify $H_2(D_n,\ZZ)\cong \ZZ/2\ZZ=\{ 0,1\}$, when $n$ is even. Then we have:

\begin{prop}\label{etale}
Let $n, g' \in \NN$ with $n\geq 3$, $g'>0$. 
Then, for any $v\in HS(D_n;g')$, we have: 
\begin{itemize}
\item[(i)] 
$v\approx (y,1,x,1,\dots,1)$, if $n$ is odd or if $n$ is even and $\e (v)=0$;
\item[(ii)] 
$v\approx (y,x^{n/2},x,1,\dots,1)$, if $n$ is even and  $\e(v) =1$.
\end{itemize}

\end{prop}

\proof
Let
$$
\ol{v}= v \, \mbox{(mod $\ZZ/n\ZZ$)} \in (\ZZ/2\ZZ)^{2g'} \, .
$$
Notice that $\ol{v}\in HS(\ZZ/2\ZZ;g')$. Since the parameter space for \'etale $\ZZ/2\ZZ$-coverings of curves of a fixed genus is 
irreducible (see e.g.  \cite{DM} Lemma 5.16, or \cite{bf},  or \cite{cor1}, or \cite{cyclic}, Thm. 2.4), there exists $\varphi \in Map_{g'}$ such that
$\varphi \cdot \ol{v}=(1,0,\,  \dots \, , 0)$. Hence 
$$
\varphi \cdot v =(x^{\ell_1}y, x^{m_1}, \, \dots \, , x^{\ell_{g'}}, x^{m_{g'}})\,  .
$$
The condition $ev(\varphi\cdot v)=1$ implies that $2m_1 =0 $ (mod $ n$). Hence $m_1 = 0$  or  $m_1 = \frac{n}{2}$ (mod $n$).\\
In the first case, which is the only possible if $n$ is odd, 
$$
\varphi \cdot v =(x^{\ell_1}y, 1, x^{\ell_2}, \, \dots \, , x^{\ell_{g'}}, x^{m_{g'}})\, .
$$ 
Consider now $v':=(x^{\ell_2}, x^{m_2}, \, \dots \, , x^{\ell_{g'}}, x^{m_{g'}})$. As $v' \in HS(\ZZ/n\ZZ; g'-1)$, from 
the irreducibility of the parameter space of \'etale $\ZZ/n\ZZ$-coverings of curves of a fixed genus we deduce that
$\exists \varphi ' \in Map_{g'-1}$ such that $\varphi ' \cdot v' =(x^\l, 1, \, \dots \, , 1)$, with $(\l, n)=1$ 
(see e.g.  \cite{DM} Lemma 5.16, or \cite{bf},  or \cite{cor1}, or \cite{cyclic}, Thm. 2.4).
Now, from Proposition \ref{auts extend} it follows that $\exists \psi \in Map_{g'}$ such that 
$$
\psi \cdot v= (x^{\ell_1}y, 1,x^\l, 1, \, \dots \, , 1) \, .
$$
We obtain the normal form (i) after operating with $Aut(D_n)$. The fact that $\e(v)=0$ follows from a standard computation
(cf. Corollary \ref{Schur cover}).\\
If $m_1= \frac{n}{2}$ (mod $n$),  we have two subcases: $\langle  x^{\ell_2}, x^{m_2}, \dots , x^{\ell_{g'}}, x^{m_{g'}}\rangle = \ZZ/n\ZZ$
(which is the case when $n/2$ is even), or $\langle  x^{\ell_2}, x^{m_2}, \dots , x^{\ell_{g'}}, x^{m_{g'}}\rangle =\langle x^2 \rangle$.
Proceeding as in the case $m_1=0$ we reach the normal form (ii) in the first subcase, otherwise
we obtain $(x^{\ell_1}y, x^{\frac{n}{2}}, x^2, 1, \, \dots \, , 1)$ in the latter subcase.
In the latter case, consider the transformation
\begin{equation}\label{map2}
(a_1,b_1,a_2,b_2) \mapsto (a_2a_1,b_1, b_1a_2b_1^{-1}, a_2b_2a_2b_1^{-1})\, ,
\end{equation}
which is realized by $Map_2$ as it preserves the relation $\prod_{1}^{2}[\a_i,\b_i]=1$.
Then extend it to $\Pi_{g'}$ using Proposition \ref{auts extend} and apply the transformation so obtained to 
$(x^{\ell_1}y, x^{\frac{n}{2}}, x^2, 1, \, \dots \, , 1)$.
We obtain:
$$
v \approx (x^{\ell_1+2}y, x^{\frac{n}{2}}, x^{2}, x^{4+\frac{n}{2}}, 1, \, \dots \, , 1) \, .
$$
Since $\ZZ/n\ZZ=\langle x^{\frac{n}{2}}, x^2\rangle$, there exists $\psi\in Map_{g'}$ such that $\psi \cdot v \approx (x^{\ell_1+2}y, x^{\frac{n}{2}}, x, 1, \, \dots \, , 1) $,
therefore we obtain the normal form (ii). In both of these subcases we have $\e(v)=1$ (cf. Corollary \ref{Schur cover}). 
  \qed

\medskip

\noindent \textbf{Case 2: $\Sigma \not=\emptyset$ and contains some reflection.}

\medskip

Let $v=(c_1, \, \dots \, , c_d ; a_1, b_1, \, \dots \, , a_{g'}, b_{g'})$ be a Hurwitz generating system such that  $\{ c_1, \, \dots \, , c_d\}$
contains some reflection, actually an even number   because any product of commutators in $D_n$ is a rotation.
If $n$ is odd, all the reflections belong to the same conjugacy class, while when $n=2k$ they are divided into two classes.
Denote by  $\nu_y$ (resp. $\nu_{xy}$) the number of $c_i$'s in the class of $y$ (resp. $xy$).
As the pair $(\nu_y,\nu_{xy})$ is not $Aut(D_n)$-invariant, we define $\nu_1, \nu_2$ by the property that 
$\{ \nu_1, \nu_2\}=\{ \nu_y,\nu_{xy}\}$, $\nu_1 \leq \nu_2$ (in \cite{CLP11} we used the notation $h$ for $\nu_1$, $k$ for $\nu_2$).
Recall that,  under the above  hypotheses, $H_{2, \Sigma}(D_n)=\{0\}$ (Corollary \ref{H2GammaD}). Indeed we prove
that all the $v$'s with fixed $g', d$, $n$ and $\nu$-type are equivalent each other. 

\begin{prop}\label{general}
Let $n, g' , d \in \NN$ with $n\geq 3$, $g' , d>0$. 
Then, for any $v\in HS(D_n;g',d)$ such that $\Sigma_v$ contains some reflection, we have: 
\begin{itemize}
\item[(i)] 
$v\approx (x^{\ul{r}}, x^{1-|\ul{r}|}y, xy, y, \, \dots \,  , y \, ; \,  x, 1, \,  \dots \,  , 1, 1)$, if $n$ is odd;
\item[(ii)]
$v\approx (x^{\ul{r}}, \underbrace{x^{\e-|\ul{r}|}y, xy,\, \dots \,  , xy}_{\nu_2}, \underbrace{ y, \, \dots \,  , y}_{\nu_1} \, ; \,  x, 1, \,  \dots \,  , 1, 1)$, if $n$ is even.
\end{itemize}
Here $\ul{r}=(r_1, \dots , r_R)$, where $ R + \nu_y = d$ in case (i), $ R + \nu_y + \nu_{xy}= d$ in case (ii),
 $0<r_i\leq r_{i+1}\leq \frac{n}{2}$, $x^{\ul{r}}=(x^{r_1}, \dots , x^{r_{ R}})$, $|\ul{r}|=\sum r_i$ (mod $n$),
$\{ \nu_1, \nu_2\}=\{ \nu_y,\nu_{xy}\}$, $\nu_1 \leq \nu_2$, $\e\in \{ 0,1\}$, $\e+\nu_2 \equiv 1  (\!\!\!\! \mod 2)$.
\end{prop}

The idea of the proof is the following. Using the action of the unpermuted mapping class group $Map^u(g',d+1)$ and the fact that at least one $c_i$ is a reflection, we 
prove that $v\sim (\tilde{c}_1 , \, \dots \, , \tilde{c}_d ; \tilde{a}_1, \tilde{b}_1, \, \dots \, , \tilde{a}_{g'}, \tilde{b}_{g'})$, with $\tilde{a}_i , \tilde{b}_i \in \ZZ/n\ZZ$, for any $i$.
We collect in the Appendix the relevant facts that will be used about the action of $Map^u(g',d+1)$ on the fundamental group.     
Then, using results about \'etale $\ZZ/n\ZZ$-covers, we deduce that  $v\approx v':=(c_1' , \, \dots \, , c_d' ; x, 1, \, \dots \, , 1)$ (Lemma \ref{v'}).
At this point we can apply the main theorem of \cite{CLP11} to deduce that, acting with the braid group,
it is possible to transform  $v'$ to the corresponding normal form. However, we will see that using the entry $x$ in $v'$, the results in the Appendix
and Lemma 2.1 of \cite{CLP11}, we can  transform directly  $v'$ in one of the above forms without using the normal 
forms for the $g'=0$ case.

\begin{lemma}\label{v'}
Let $v$ be as in Proposition \ref{general}. Then
$$
v\approx v':=(c_1' ,\,  \dots \,  , c_d' \, ; \,  x, 1, \dots , 1, 1)  \, .
$$
\end{lemma}
\proof 
Without loss of generality assume that $c_d$ is a reflection (otherwise act with the braid group).
Then,  if $a_1$ is a reflection, by Proposition \ref{xi-twists eqs} (i), $\exists \varphi \in Map^u(g',d+1)$ such that
$$
\varphi \cdot v = (c_1, \, \dots \, , c_{d-1}, (c_da_1b_1a_1^{-1})c_d(c_da_1b_1a_1^{-1})^{-1} ; c_da_1, b_1, \, \dots \, , a_{g'}, b_{g'}) \, .
$$
While, if $a_1$ is a rotation and $b_1$ is a reflection, by Proposition \ref{xi-twists eqs} (ii) we have: 
$\exists \varphi \in Map^u(g',d+1)$ such that
$$
\varphi \cdot v  =   (c_1,  \dots  , (c_d[a_1,b_1]a_1^{-1})c_d(c_d[a_1,b_1]a_1^{-1})^{-1} ; a_1, (a_1^{-1}c_da_1)b_1,  \dots  , b_{g'})  .
$$
Notice that in both cases the $d$-th entry of $\varphi \cdot v$ is a reflection and that $c_da_1, (a_1^{-1}c_da_1)b_1\in \ZZ/n\ZZ$.
Proceeding in this way we get $\psi \in Map^u(g',d+1)$ such that $(\psi \cdot v)_i \in \ZZ/n\ZZ$, $i=d+1, \, \dots \, , 2g'$. 

Next, by the main theorem in  \cite{cyclic}, we conclude that $\psi\cdot v \approx (\tilde{c}_1, \,  \dots \,  , \tilde{c}_{d};x^\alpha, 1,\,  \dots \,  ,1 )$.
We can further assume   $({\alpha},n)=1$.  Otherwise, since $D_n=\langle x^\alpha , {\tilde{c}}_1,  \dots , {\tilde{c}}_{d} \rangle$,
there exists $x^\beta \in \langle {\tilde{c}}_1,  \dots , {\tilde{c}}_{d} \rangle$ such that $\ZZ/n\ZZ=\langle x^{\alpha + \beta}\rangle$.
Using Proposition \ref{xi-twists eqs} (i) and the braid group, we can multiply $x^\alpha$ by any element of  $ \langle {\tilde{c}}_1,  \dots , {\tilde{c}}_{d} \rangle$.
The claim now  follows by applying $Aut (D_n)$. \qed 

\medskip

We now complete the proof of Proposition \ref{general}. Let $v'$ be as in Lemma \ref{v'} and let  $2N$ be the number of reflections in $\{ c_1', \dots , c_d' \}$. Applying Lemma 2.1 of \cite{CLP11} we have
\begin{equation}\label{v'1}
{v}' \approx (x^{\ul{r}}, x^{\beta}y, x^{\alpha}y, x^{j_{N-1}}y, x^{j_{N-1}}y, \dots , x^{j_1}y, x^{j_1}y \, ; \,  x, 1, \dots , 1) \, ,
\end{equation}
where $\ul{r}=(r_1, \dots , r_R)$, $0<r_i\leq r_{i+1}\leq \frac{n}{2}$, $x^{\ul{r}}=(x^{r_1}, \dots , x^{r_R})$.

If $N=1$ the result is clear. Otherwise we  conjugate by $x$ simultaneously  the entries of each pair    $( x^{j_{k}}y, x^{j_{k}}y)$ in \eqref{v'1} 
without changing the other components, hence we obtain:
\begin{align}\label{eq**}
 {v}'  \quad  \sim  \quad
(x^{\ul{r}}, x^{\beta}y, x^{\alpha}y,& x^{j_{N-1}+2\ell_{N-1}}y, x^{j_{N-1}+2\ell_{N-1}}y, \dots \\
& \dots, x^{j_1+2\ell_{1}}y, x^{j_1+2\ell_{1}}y\, ; \,  x,  1, \dots , 1) \nonumber
\end{align}
for any $\ell_1 , \dots , \ell_{N-1}\in \ZZ$.  \\
The equivalence \eqref{eq**} can be proven as follows. We have:
\begin{align*}
 {v}' & \sim ({x^{\ul{r}}, x^{\beta}y,  \dots , x^{j_1}y, (x^{j_1}y x^{-1})x^{j_1}y(x^{j_1}y x^{-1})^{-1}}\, ;\, x, x^{-1} x^{j_1}y x, 1, \dots , 1)\\
 & \sim ({x^{\ul{r}}, x^{\beta}y,  \dots , ( x^{-1})x^{j_1}y( x) , x^{j_1}y}\, ; \,  x, x^{-1} x^{j_1}y x, 1, \dots , 1)\\
& \sim ({x^{\ul{r}}, x^{\beta}y,  \dots , ( x^{-1})x^{j_1}y( x) , (x^{-1})x^{j_1}y(x)}\, ; \,  x,  1, \dots , 1)\, ,
\end{align*}
where the first and the third equivalences are given by $\xi$-twists as in Proposition \ref{xi-twists eqs} (ii), while the second is a braid twist between the last two components.
Iterating these steps we  can conjugate by any power of $x$   the entries of  $( x^{j_{1}}y, x^{j_{1}}y)$ simultaneously.
By Lemma 2.3 in \cite{CLP11} we can move $( x^{j_{k}}y, x^{j_{k}}y)$ to the right and then conjugate its entries 
by any power of  $x$ as before. This proves \eqref{eq**}. 

If $n$ is odd, choose $\ell_i$ in \eqref{eq**} such that $j_i +2\ell_i= \alpha -1$ (mod $n$), then 
apply the automorphism $x^{\alpha -1}y \mapsto y, x\mapsto x$ to obtain  (i). 

Assume now that $n$ is even. Without loss of generality we have that $\nu_1=\nu_y\leq \nu_{xy}=\nu_2$
(otherwise apply $Aut(D_n)$).  Assume further that $x^{\beta}y$ 
in \eqref{v'1} is conjugate to $xy$.  \\
If  $x^{\alpha}y$ is conjugate to $xy$,  choose $\ell_i$  such that $j_i+2\ell_i = \alpha$ or $j_i+2\ell_i = \alpha -1$ (mod $n$),
so \eqref{eq**} becomes:
$$
{v}' \sim ({x^{\ul{r}}, x^{\beta}y, x^{\alpha}y, x^{\alpha}y, \dots , x^{\alpha}y,  x^{\alpha - 1 }y, \dots ,  x^{\alpha -1}y}\, ; \, x, 1, \dots , 1) \, .
$$
We obtain the normal form (ii) after applying the automorphism $x^\a y \mapsto xy$, $x\mapsto x$.\\
The remaining case, where $x^{\alpha}y$ is conjugate to $y$, is similar.
  \qed
  
  \medskip

Notice that \eqref{eq**} follows also from  Lemma 2.1 in \cite{Ka2} (see also \cite{GHS}), which applies to a more general situation.
Since we don't need the whole strength of that result, we preferred to give a complete proof in our case.

\noindent \textbf{Case 3: $\Sigma \not=\emptyset$ and does not contain reflections.}

\medskip

We prove the following 

\begin{prop}
Let $v, v' \in HS(D_n;g', d)$ with $\Sigma_v , \Sigma_{v'} \subset \ZZ/n\ZZ$.
Then $v\approx v'$ if and only if there exists $f\in Aut(D_n)$ such that $f(\Sigma_v)=\Sigma_{v'}$
and $f(\e(v))=\e (v')$.
\end{prop}

The ``only if" part is clear. So assume $\Sigma_v = \Sigma_{v'}=:\Sigma$ and the existence of $f$ as in the statement. 
We prove that $v\approx v'$. This is achieved after considering three cases: $n$ is odd; $n=2k$ and $x^k\in \Sigma$,
$n=2k$ and $x^k\not\in \Sigma$. In the first two cases we determine a normal form, with respect to $\approx$,
 for each such element of $HS(D_n;g',d)$,
and then we show that two such elements are equivalent if and only if they have the same normal form. Notice that in both cases
$H_{2,\Sigma}(D_n)=\{ 0\}$ (Corollary \ref{H2GammaD}).
In the last case, for any such $v\in HS(D_n;g',d)$, we determine a normal form $v'$, with respect to the action of $Map(g',d)$ and then we show that
$v_1\approx v_2$ if and only if $\exists f\in Aut(D_n)$ such that $f(\Sigma)=\Sigma$ and $f(\e(v_1'))=\e (v_2')$.
Notice that, in this case $H_{2,\Sigma}(D_n)\cong \ZZ/2\ZZ$ (Corollary \ref{H2GammaD}).

\medskip

\noindent \textbf{Notation.} Let $v=(c_1, \dots , c_d \, ; \, a_1,b_1, \dots , a_{g'},b_{g'}) \in D_n^{d+2g'}$ be a Hurwitz generating system
such that $\Sigma_v \subset \ZZ/n\ZZ$, i.e. $c_i \in \ZZ/n\ZZ$, $\forall i$.
We denote by $H=\langle c_1, \dots , c_d \rangle \subset D_n$ the subgroup generated by the $c_i$'s.
Note that, under the above hypotheses,  $H$ is normal and it is contained in $\ZZ/n\ZZ$.
Set $G':=D_n/H$. Then $G'$ is a dihedral group $D_m$, $m\geq 3$, or is isomorphic to $\ZZ/2\ZZ \times \ZZ/2\ZZ$,
or to $\ZZ/2\ZZ$.

\begin{lemma}
Let $n\in \NN$, $n\geq 3$ odd. Let $v\in HS(D_n;g',d)$ with $\Sigma_v \subset \ZZ/n\ZZ$.
Then 
$$
v\approx (x^{\ul{r}};y,x^{h},x, 1 , \dots , 1) \, ,
$$ 
where $\ul{r}=(r_1, \dots , r_d)$, $x^{\ul{r}}=(x^{r_1}, \dots , x^{r_d})$,
$r_1 \leq \,  \dots \,  \leq r_d <\frac{n}{2}$ and  $2h= \sum_1^d r_i, \, (\!\!\!\! \mod n)$.
\end{lemma}
\proof
Let us consider 
$$
\ol{v} := (\ol{a_1}, \ol{b_1} , \dots , \ol{a_{g'}}, \ol{b_{g'}})   \in HS(G';g',0) \, ,
$$
where  $\ol{a_i}= a_i$ (mod $H$), $\ol{b_i}= b_i$ (mod $H$). By Proposition \ref{etale} and by the analogous results
for cyclic and ($\ZZ/2\ZZ\times \ZZ/2\ZZ$)-covers we have:
$$
\exists \varphi \in Map_{g'} \quad \mbox{such that } \quad \varphi \cdot \ol{v} =(\ol{y}, 1, \ol{x}, 1, \dots , 1) \, .
$$
By Proposition \ref{auts extend}, $\exists \tilde{\varphi} \in Map(g',d)$ with  
$$
\tilde{\varphi} \cdot v = (c_1, \dots , c_d \, ; \, x^{\ell_1}y, x^{m_1}, x^{\ell_2}, \dots , x^{m_{g'}}) \, ,
$$
 where 
$x^{m_i}\in H$, $\forall i$,  $x^{\ell_2} = x$ (mod $H$) and $x^{\ell_i} \in H$, $\forall i>2$. 

We now apply the $\xi$-twists as in  Proposition \ref{xi-twists eqs} (i)  with $\ell =2,3, \dots , g'$
and we deduce that we can multiply all the $x^{\ell_i}$, $i>1$, by any element of $H$. Hence $\exists \psi \in Map^u(g',d+1)$ such that
$$
\psi \cdot \tilde{\varphi} \cdot v= (c_1, \dots , c_d \, ; \, x^{\ell_1}y, x^{m_1}, x, x^{m_2}, 1, x^{m_3},  \dots ,1, x^{m_{g'}})\,  . 
$$
Similarly, using Proposition \ref{xi-twists eqs} (ii), we get: 
$$
v \approx (c_1, \dots , c_d \, ; \, x^{\ell_1}y, x^{m_1}, x,  1,   \dots ,1, 1)\, .
$$

Now, for any $i=1, \dots , d$, consider $c_i=x^{s_i}$. If $s_i<\frac{n}{2}$, set $r_i=s_i$, otherwise use the braid group
to move $c_i$ to the $d$-th position and then apply Proposition \ref{xi-twists eqs} (ii)
with $\ell =1$. After this, $c_i$ becomes $c_i^{-1}=x^{n-s_i}$, then set $r_i=n-s_i$.   Finally, using the braid group,
we can order the $c_i$'s such that $r_i\leq r_{i+1}$.

So, we have proved that
$$
v\sim (x^{\ul{r}};x^{\l_1}y,x^{\mu_1},x, 1 , \dots , 1) \, ,
$$
with $r_1 \leq \,  \dots \,  \leq r_d <\frac{n}{2}$. Now the condition $ev(v)=1$ implies that ${2\mu_1}=\sum_1^d r_i $ (mod $n$),
therefore set $h:=\mu_1$ (mod $n$). 

We reach the normal form after applying the automorphism $x^{\lambda_1}y\mapsto y$, $x\mapsto x$.
\qed

\begin{lemma}
Let $n=2k\in \NN$ and let $v \in HS(D_n;g',d)$ with $x^k \in \Sigma_v \subset \ZZ/n\ZZ$.
Then we have:
$$
v\approx (x^{\ul{r}};y,x^{h},x, 1 , \dots , 1) \, ,
$$
where $\ul{r}=(r_1, \dots , r_d)$, $x^{\ul{r}}=(x^{r_1}, \dots , x^{r_d})$,
$r_1 \leq \,  \dots \,  \leq r_d = k$,  $2h= \sum_1^d r_i$ (mod $n$) and $h< k$.
\end{lemma}
\proof
Proceeding as in the proof of the previous lemma,  we have:
$$
\exists \varphi \in Map(g',d) \quad \mbox{such that }  \, \varphi \cdot v = (c_1, \dots , c_d \, ; \, x^{\ell_1}y, x^{m_1}, x^{\ell_2}, \dots , x^{m_{g'}}) \, ,
$$
where $x^{m_i}\in H$, $\forall i>1$,  $x^{\ell_2} = x$ (mod $H$) and $x^{\ell_i} \in H$, $\forall i>2$. \\
Since we can multiply all the $x^{\ell_i}$, $i>1$, by any element of $H$ (apply Proposition  \ref{xi-twists eqs} (i)   with $\ell =2,3, \dots , g'$),
we have: $\exists \psi \in Map^u(g',d+1)$ such that
$$
\psi \cdot \varphi \cdot v= (c_1, \dots , c_d \, ; \, x^{\ell_1}y, x^{m_1}, x, x^{m_2}, 1, x^{m_3},  \dots ,1, x^{m_{g'}})\,  . 
$$
Similarly, using Proposition \ref{xi-twists eqs} (ii), we get: 
$$
v \approx (c_1, \dots , c_d \, ; \, x^{\ell_1}y, x^{m_1}, x,  1,   \dots ,1, 1)\, .
$$

Now, for any $i=1, \dots , d$, consider $c_i=x^{s_i}$. If $s_i \leq k$, set $r_i=s_i$, otherwise use the braid group
to move $c_i$ to the $d$-th position and then apply Proposition \ref{xi-twists eqs} (ii)
with $\ell =1$. In this way $c_i$ becomes $c_i^{-1}$ and so set $r_i=2k-s_i$. 

So, we have proved that
$$
v\approx (x^{\ul{r}};x^{\l_1}y,x^{\mu_1},x, 1 , \dots , 1) \, ,
$$
with $r_i \leq k$, $\forall i$. Now the condition $ev(v)=1$ implies that ${2\mu_1}=\sum_1^d r_i $ (mod $n$).
If $\mu_1<k$,  set $h=\mu_1$. Otherwise, apply braid group transformations  to achieve the ordering $r_i\leq r_{i+1}$, $\forall i\leq d-1$.
By hypotheses $r_d=k$ and we   apply Proposition 
\ref{xi-twists eqs} with $\ell =1$. Since $x^k$ is central, this operation does not change $r_d$, while $x^{\mu_1}$ becomes $x^{\mu_1+k}$.
Set $h=\mu_1 + k $ (mod $n$). 

Finally apply the appropriate element of  $Aut(D_n)$ to reach the normal form.
\qed

\medskip

We now consider the last case.
\begin{lemma}\label{bad case}
Let $n=2k\in \NN$ and let $v \in HS(D_n;g',d)$ with $\Sigma \subset \ZZ/n\ZZ \setminus \{ x^k \}$.
Then we have:
\begin{itemize}
\item[(i)]
$v \approx  v' := (x^{\ul{r}}; y,x^{h},x, 1 , \dots , 1)$,
where $\ul{r}=(r_1, \dots , r_d)$, $x^{\ul{r}}=(x^{r_1}, \dots , x^{r_d})$,
$r_1 \leq \,  \dots \,  \leq r_d < k$,  $2h= \sum_1^d r_i$ $(\!\!\!\! \mod n)$;
\item[(ii)]
let  $v_1'=(x^{\ul{r}}; y,x^{h},x, 1 , \dots , 1)$ and $v_2'=(x^{\ul{r}}; y,x^{h+k},x, 1 , \dots , 1)$,
then $\e({v_1'})\not= \e({v_2'})\in (D_n)_\Sigma$;
\item[(iii)]
$v_1'\approx v_2'$ if and only if $\exists f\in Aut(D_n)$ such that $f(\Sigma)=\Sigma$ and $f(\e(v_1'))=\e (v_2')$.
\end{itemize}
\end{lemma}
\proof
The proof of (i) is the same as that of the previous lemma. Since in this case $x^k \not\in \Sigma$, 
we can not achieve $h\leq k$.

To prove (ii) recall that $\e (v):=ev(\hat{v})\in (D_n)_\Sigma$. So, if $ev(\widehat{v_1'})=ev(\widehat{v_2'})$, 
then $ev(\widehat{v_2'})^{-1}\cdot ev(\widehat{v_1'})=0 \in H_{2, \Sigma}(D_n)$.
But now a direct computation shows that 
$ev(\widehat{v_2'})^{-1}\cdot ev(\widehat{v_1'})\not= 0$ (Corollary \ref{Schur cover}), a contradiction.

(iii) The ``only if" part is clear. 
So, assume that $\exists f\in Aut(D_n)$ such that $f(\Sigma)=\Sigma$ and $f(\e(v_1'))=\e (v_2')$.
Since $f(\e(v_1'))=\e(f(v_1'))$, we have $\e(f(v_1'))=\e(v_2')$ and so $v_2'$ and $f(v_1')$ have the same 
$\nu$-type (Remark \ref{evsnu}).
>From (i) and (ii) we deduce that 
$$
f(v_1')\approx (x^{\ul{r}};x^{\l_1} y,x^{h+k},x, 1 , \dots , 1) \, .
$$
 Hence,
using the automorphism $x^{\l_1}y\mapsto  y$, $x\mapsto x$, we have that $f(v_1')\approx v_2'$ and so the claim follows.

\qed

\section{Appendix A: Automorphisms of surface-groups}
We collect in this Appendix some  facts about mapping class groups and their action 
on fundamental groups. They  should be  well known to experts, we include them here 
for completeness.

Let $Y$ be a compact Riemann surface of genus $g'$ and let $\mathcal{B}=\{ y_1, \dots , y_d\}\subset Y$
be a finite subset  of cardinality $d$. After the choice of a geometric basis of 
 $Y\setminus \mathcal{B}$, we have the following presentation of the 
 fundamental group:
\begin{align*}
& \pi_1(Y\setminus \mathcal{B}, y_0)=\\
&\langle \,  \gamma_1\, , \, \dots \, , \, \gamma_d \, , \, \alpha_1\, , \, \beta_1\, , \, \dots \, ,
\, \alpha_{g'}\, , \, \beta_{g'}
 \, | \, 
\gamma_1\cdot \, \dots \,
\cdot \gamma_d \cdot \Pi_{i=1}^{g'}[\, \alpha_i\, , \, \beta_i\, ]
=1\, \rangle \, .
\end{align*}

Following \cite{Bir69}, there is a short exact sequence 
\begin{equation}\label{mapout}
1\to \pi_1(Y\setminus \mathcal{B}, y_0) \xrightarrow{\Xi} Map^u(Y,\{ y_0,y_1,\dots , y_d\})\to Map^u(Y,\mathcal{B})\to 1
\end{equation}
which induces an injective group homomorphism 
$$
Map^u(Y,\mathcal{B}) \to Out\left( \pi_1(Y\setminus \mathcal{B}, y_0)\right)
$$ 
(\cite{Bir69}, Thm. 4).
The map $\Xi$ is defined  as follows. Let   $[c]\in  \pi_1(Y\setminus \mathcal{B}, y_0) $ be an element of the  geometric basis
and let $c\colon [0,2\pi]\to Y\setminus \mathcal{B}$ be a simple, smooth loop based at $y_0$,
representing $[c]$. 
Let $\Xi([c])$ be the isotopy class of the $\xi$-twist, $\xi_c$. Then extend $\Xi$ to the whole group as an homomorphism.
Recall that the $\xi$-twist, $\xi_c$,  can be defined as follows.  Let $N\subset Y\setminus \mathcal{B}$ be a tubular neighborhood of $c$ and let 
$
e\colon A \to N 
$
be a diffeomorphism between the annulus $A=\{ z=re^{i\theta}\in \mathbb{C}\, | \, 1\leq r \leq 2\}$ and $N$
such that $e(\frac{3}{2}, \theta)=c(\theta)$. Define $h\colon A\to A$ as follows
$$
h(r,\theta)=\begin{cases} (r,\theta +4\pi(r-1))\, , \, 1\leq r \leq \frac{3}{2};\\
(r,\theta +4\pi(2-r))\, , \, \frac{3}{2}\leq r \leq 2\, .\end{cases}
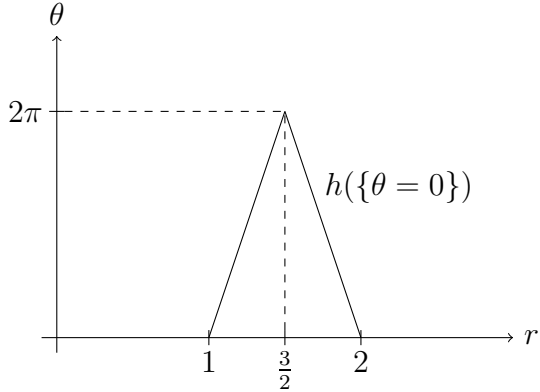
$$

\begin{tikzpicture}
\draw[->] (-0.2,0) -- (6,0) node[right] {$r$}; 
\draw[->] (0,-0.2) -- (0,4) node[above] {$\theta$}; 

\draw (2,-.1) -- (2,.1) node[below=4pt] {$1$};
\draw (3,-.1) -- (3,.1) node[below=4pt] {$\frac{3}{2}$};
\draw (4,-.1) -- (4,.1) node[below=4pt] {$2$};

\draw (-.1,3) -- (.1,3) node[left=4pt] {$2\pi$};

\draw[dashed] (.1,3) -- (3,3);
\draw[dashed] (3,.1) -- (3,3);

\draw (2,0) -- (3,3);
\draw (3,3) -- (4,0);
\node at (4.5,2) {$h(\{\theta =0\})$};

\end{tikzpicture}

Then $h$ is a diffeomorphism which is the identity when $r=1, \frac{3}{2} , 2$. Finally, define  $\xi_c\colon Y \to Y$ 
as the identity on
$Y\setminus N$ and as $e\circ h\circ e^{-1} $ on $N$.

>From the sequence \eqref{mapout}, it follows that  $\pi_1(Y\setminus \mathcal{B}, y_0)$ is isomorphic
through $\Xi$ to a normal subgroup of $Map^u(Y,\{ y_0,y_1,\dots , y_d\})$, hence we get an action by conjugation 
of $Map^u(Y,\{ y_0,y_1,\dots , y_d\})$ 
on $\pi_1(Y\setminus \mathcal{B}, y_0)$:
$$
[f]\cdot [\xi_c]=[f\circ \xi_c \circ f^{-1}]\, .
$$
We have:
\begin{lemma}\label{mapaut}
For any $[f]\in Map^u(Y,\{ y_0,y_1,\dots , y_d\})$ and $[c]\in \pi_1(Y\setminus \mathcal{B}, y_0)$, we have:
$$
[f]\cdot [\xi_c]=[\xi_{f_{\#}(c)}]\, ,
$$
where $f_{\#}(c)(\theta)=(f\circ c)( \theta)$.
\end{lemma}
\begin{proof}
$f\circ \xi_c \circ f^{-1}=(f\circ e)\circ h \circ (f\circ e)^{-1}$ on $N$, and coincides with the identity on $Y\setminus N$.
The result then follows because $f\circ e\colon A \to Y$ is a tubular neighborhood of $f_{\#}(c)$.
\end{proof}

One can define, in the same way, $\xi$-twists with respect to loops that are not based at $y_0$
and Lemma \ref{mapaut} is still valid.
In the following result we give the action of $\xi$-twists around certain loops in terms of 
a given geometric basis of $\pi_1(Y\setminus \mathcal{B}, y_0)$.

\begin{prop}\label{xi-twists eqs}
Let $ \gamma_1\, , \, \dots \, , \, \gamma_d \, , \, \alpha_1\, , \, \beta_1\, , \, \dots \, ,\, \alpha_{g'}\, , \, \beta_{g'}$ be a fixed geometric basis
of  $\pi_1(Y\setminus \mathcal{B}, y_0)$. 
\begin{itemize}
\item[(i)]
Let $c \subset Y\setminus \mathcal{B}$ be the loop in Figure 1, image of the two sides of the angle inside the polygon with vertex $y_d$. Set $u=\prod_{k=1}^{\ell -1}[\alpha_k,\beta_k]$.
 Then we have:
\begin{eqnarray*}
(\xi_{c})_*(\alpha_\ell)&=&u^{-1}\gamma_d u \alpha_\ell\, ; \\
(\xi_{c})_*(\gamma_d)&=& \left( \gamma_d u\alpha_\ell \beta_\ell \alpha_\ell^{-1}u^{-1}\right)
\gamma_d\left( \gamma_d u\alpha_\ell \beta_\ell \alpha_\ell^{-1}u^{-1}\right)^{-1}\, ;
\end{eqnarray*}
$(\xi_{c})_*(\alpha_i)=\alpha_i$ ($i\not= \ell$), $(\xi_{c})_*(\beta_i)=\beta_i$ ($\forall \, i$), $(\xi_{c})_*(\gamma_j)=\gamma_j$ ($j\not=d$).
\item[(ii)]
Let $c  \subset Y\setminus \mathcal{B}$ be the loop in Figure 2,  image of the two sides of the angle inside the polygon with vertex $y_d$. Set $u=\prod_{k=1}^{\ell -1}[\alpha_k,\beta_k]$. Then we have:
\begin{eqnarray*}
(\xi_{c})_*(\beta_\ell) &=& \alpha_\ell^{-1}u^{-1}\gamma_d u\alpha_\ell\beta_\ell \, ; \\
(\xi_{c})_*(\gamma_d)&=& \left( \gamma_d u[\alpha_\ell,\beta_\ell]\alpha_\ell^{-1}u^{-1}\right)
\gamma_d\left( \gamma_d u[\alpha_\ell,\beta_\ell]\alpha_\ell^{-1}u^{-1}\right)^{-1}\, ;
\end{eqnarray*}

$(\xi_{c})_*(\beta_i)=\beta_i$ ($i\not= \ell$), $(\xi_{c})_*(\alpha_i)=\alpha_i$ ($\forall \, i$), $(\xi_{c})_*(\gamma_j)=\gamma_j$ ($j\not=d$).
\end{itemize} 
\end{prop}
\begin{proof}
(i) The image of $\alpha_\ell$ under $\xi_c$ is drawn in Figure 3. From this it follows the formula for $(\xi_c)_*(\alpha_\ell)$.
Since $\xi_c$ is the identity outside a small tubular neighborhood of $c$, we have that
$(\xi_{c})_*(\alpha_i)=\alpha_i$ ($ i\not= \ell$), $(\xi_{c})_*(\beta_i)=\beta_i$ ($\forall \, i$) and $(\xi_{c})_*(\gamma_j)=\gamma_j$ ($j\not=d$).
The formula for $(\xi_c)_*(\gamma_d)$ is now a consequence of $\gamma_1\cdot\,  \dots \,  \cdot \gamma_d\prod_{i=1}^{g'}[\alpha_i,\beta_i]=1$,
since the product $\gamma_1\cdot\,  \dots \,  \cdot \gamma_d\prod_{i=1}^{g'}[\alpha_i,\beta_i]$ must be left fixed.

The proof of (ii) is similar.
\end{proof}

\begin{tikzpicture} 
\path (0,3) coordinate (origin); 
\path (-75:2.5cm) coordinate (P0); 
\path (-60:2.41cm) coordinate (P01); 
\path (-45:2.5cm) coordinate (P1); 
\path (-15:2.5cm) coordinate (P2); 
\path (15:2.5cm) coordinate (P3); 
\path (45:2.5cm) coordinate (P4); 
\path (60:2.41cm) coordinate (P45); 
\path (75:2.5cm) coordinate (P5);
\path (90:2.41cm) coordinate (P56);
\path (105:2.5cm) coordinate (P6);
\path (120:2.41cm) coordinate (P67);
\path (135:2.5cm) coordinate (P7);
\path (150:2.41cm) coordinate (P78);
\path (165:2.5cm) coordinate (P8);
\path (195:2.5cm) coordinate (P9);
\path (225:2.5cm) coordinate (P10);
\path (255:2.5cm) coordinate (P11);
\path (270:2.41cm) coordinate (P110);
\draw[->] (P0) -- (P01);
\draw (P01) -- (P1);
\node at (-60:2.8cm) {$\alpha_1$};
\draw[dashed] (P1) -- (P2) -- (P3) -- (P4);
\draw[->] (P4) -- (P45);
\draw (P45) -- (P5);
\draw[->] (P5) -- (P56);
\draw (P56) -- (P6);
\draw (P6) -- (P67);
\draw[<-] (P67) -- (P7);
\draw (P7) -- (P78);
\draw[<-] (P78) -- (P8);
\draw[dashed] (P8) -- (P9) -- (P10) -- (P11); 
\draw (P11) -- (P110);
\draw[<-] (P110) -- (P0);

\node at (270:2.8cm) {$\beta_{g'}$};
\node at (60:2.8cm) {$\alpha_{\ell}$};
\node at (90:2.8cm) {$\beta_{\ell}$};
\node at (120:2.8cm) {$\alpha_{\ell}$};
\node at (150:2.8cm) {$\beta_{\ell}$};
\node at (-75:2.8cm) {$y_0$};

\draw (P0) -- (-75:1.1cm);

\node at (-75:0.5cm) {$\bullet$};
\node at (-105:0.6cm) {$y_d$};

\draw[
        decoration={markings, mark=at position 0.625 with {\arrow{<}}},
        postaction={decorate}
        ]
        (-75:0.5) circle (0.6);
        
\draw[
        decoration={markings, mark=at position 0.625 with {\arrow{<}}},
        postaction={decorate}
        ] (-75:0.5) -- (56:2.42);
        
\node at (30:1.2cm) {$c$};

\draw[
        decoration={markings, mark=at position 0.625 with {\arrow{>}}},
        postaction={decorate}
        ] (-75:0.5) -- (124:2.42);

\node at (-30:1.2cm) {$\gamma_d$};

\node at (-90:4cm) {Figure $1$.};

\end{tikzpicture} 
\hfill
\begin{tikzpicture} 
\path (-2,3) coordinate (origin); 
\path (-75:2.5cm) coordinate (P0); 
\path (-60:2.41cm) coordinate (P01); 
\path (-45:2.5cm) coordinate (P1); 
\path (-15:2.5cm) coordinate (P2); 
\path (15:2.5cm) coordinate (P3); 
\path (45:2.5cm) coordinate (P4); 
\path (60:2.41cm) coordinate (P45); 
\path (75:2.5cm) coordinate (P5);
\path (90:2.41cm) coordinate (P56);
\path (105:2.5cm) coordinate (P6);
\path (120:2.41cm) coordinate (P67);
\path (135:2.5cm) coordinate (P7);
\path (150:2.41cm) coordinate (P78);
\path (165:2.5cm) coordinate (P8);
\path (195:2.5cm) coordinate (P9);
\path (225:2.5cm) coordinate (P10);
\path (255:2.5cm) coordinate (P11);
\path (270:2.41cm) coordinate (P110);
\draw[->] (P0) -- (P01);
\draw (P01) -- (P1);
\node at (-60:2.8cm) {$\alpha_1$};
\draw[dashed] (P1) -- (P2) -- (P3) -- (P4);
\draw[->] (P4) -- (P45);
\draw (P45) -- (P5);
\draw[->] (P5) -- (P56);
\draw (P56) -- (P6);
\draw (P6) -- (P67);
\draw[<-] (P67) -- (P7);
\draw (P7) -- (P78);
\draw[<-] (P78) -- (P8);
\draw[dashed] (P8) -- (P9) -- (P10) -- (P11); 
\draw (P11) -- (P110);
\draw[<-] (P110) -- (P0);

\node at (270:2.8cm) {$\beta_{g'}$};
\node at (60:2.8cm) {$\alpha_{\ell}$};
\node at (90:2.8cm) {$\beta_{\ell}$};
\node at (120:2.8cm) {$\alpha_{\ell}$};
\node at (150:2.8cm) {$\beta_{\ell}$};
\node at (-75:2.8cm) {$y_0$};

\draw (P0) -- (-75:1.1cm);

\node at (-75:0.5cm) {$\bullet$};
\node at (-105:0.6cm) {$y_d$};

\draw[
        decoration={markings, mark=at position 0.625 with {\arrow{<}}},
        postaction={decorate}
        ]
        (-75:0.5) circle (0.6);
        
\draw[
        decoration={markings, mark=at position 0.625 with {\arrow{<}}},
        postaction={decorate}
        ] (-75:0.5) -- (80:2.45);
        
\node at (60:1.2cm) {$c$};

\draw[
        decoration={markings, mark=at position 0.625 with {\arrow{>}}},
        postaction={decorate}
        ] (-75:0.5) -- (160:2.45);

\node at (-30:1.2cm) {$\gamma_d$};

\node at (-90:4cm) {Figure $2$.};

\end{tikzpicture} 

\begin{tikzpicture} 
\path (0,3) coordinate (origin); 
\path (-75:2.5cm) coordinate (P0); 
\path (-60:2.41cm) coordinate (P01); 
\path (-45:2.5cm) coordinate (P1); 
\path (-15:2.5cm) coordinate (P2); 
\path (15:2.5cm) coordinate (P3); 
\path (45:2.5cm) coordinate (P4); 
\path (60:2.41cm) coordinate (P45); 
\path (75:2.5cm) coordinate (P5);
\path (90:2.41cm) coordinate (P56);
\path (105:2.5cm) coordinate (P6);
\path (120:2.41cm) coordinate (P67);
\path (135:2.5cm) coordinate (P7);
\path (150:2.41cm) coordinate (P78);
\path (165:2.5cm) coordinate (P8);
\path (195:2.5cm) coordinate (P9);
\path (225:2.5cm) coordinate (P10);
\path (255:2.5cm) coordinate (P11);
\path (270:2.41cm) coordinate (P110);
\draw[->] (P0) -- (P01);
\draw (P01) -- (P1);
\node at (-60:2.8cm) {$\alpha_1$};
\draw[dashed] (P1) -- (P2) -- (P3) -- (P4);
\draw[->] (P4) -- (P45);
\draw (P45) -- (P5);
\draw[
        decoration={markings, mark=at position 0.5 with {\arrow{>}}},
        postaction={decorate} 
        ] (P4) --(-85:1cm) node[right] {$(\xi_{c})_* (\alpha_\ell)$};
\draw (-85:1cm) -- (124:2.42cm);
\draw (124:2.42cm) -- (0:0cm);
\draw[
        decoration={markings, mark=at position 0.5 with {\arrow{>}}},
        postaction={decorate} 
        ] (0:0cm) -- (P5);

\draw[->] (P5) -- (P56);
\draw (P56) -- (P6);
\draw (P6) -- (P67);
\draw[<-] (P67) -- (P7);
\draw (P7) -- (P78);
\draw[<-] (P78) -- (P8);
\draw[dashed] (P8) -- (P9) -- (P10) -- (P11); 
\draw (P11) -- (P110);
\draw[<-] (P110) -- (P0);

\node at (270:2.8cm) {$\beta_{g'}$};
\node at (60:2.8cm) {$\alpha_{\ell}$};
\node at (90:2.8cm) {$\beta_{\ell}$};
\node at (120:2.8cm) {$\alpha_{\ell}$};
\node at (150:2.8cm) {$\beta_{\ell}$};
\node at (-75:2.8cm) {$y_0$};

\node at (-75:0.5cm) {$\bullet$};
        
\draw[
        decoration={markings, mark=at position 0.625 with {\arrow{<}}},
        postaction={decorate}
        ] (-75:0.5) -- (56:2.42);

\draw[
        decoration={markings, mark=at position 0.625 with {\arrow{>}}},
        postaction={decorate}
        ] (-75:0.5) -- (124:2.42);

\node at (-90:4cm) {Figure $3$.};

\end{tikzpicture} 

\begin{prop}\label{auts extend}
Let $\Pi_{g'}=\langle \a_1, \dots , \b_{g'}\, | \, \prod_1^{g'}[\a_i,\b_i]\rangle$ and
 $\Pi_{g'-1}=\langle \a_2, \dots , \b_{g'}\, | \, \prod_2^{g'}[\a_i,\b_i]\rangle$. Then, 
for any $\varphi \in Aut^0(\Pi_{g'-1})$, there exists ${\psi}\in Aut^0(\Pi_{g'})$ and $\d\in \Pi_{g'}$
such that ${\psi}(\a_1)=\a_1$, ${\psi}(\b_1)=\b_1$, ${\psi}(\a_i)=\d \varphi(\a_i)\d^{-1}$,
${\psi}(\b_i)=\d \varphi(\b_i)\d^{-1}$, $i>1$.
\end{prop}
\proof
We first extend $\varphi$ to an automorphism 
$$
\tilde{\varphi} \in Aut\left( \langle \a_2, \dots , \b_{g'}, \ga \, | \, \ga \cdot \prod_2^{g'}[\a_i,\b_i]\rangle \right)
$$
such that $\tilde{\varphi}(\a_i)={\varphi}(\a_i)$, $\tilde{\varphi}(\b_i)={\varphi}(\b_i)$ and $\tilde{\varphi} (\ga)= \d^{-1}\ga \d $, $i>1$. Geometrically this corresponds to representing
$\varphi$ as composition of Dehn twists along curves contained in the complement $Y_{g'-1}\setminus D$ of a closed disk
$D$ in a Riemann surface $Y_{g'-1}$ of genus $g'-1$, where $D$ does not intersect $\a_i$ and $\b_i$. 

 Now simply define $\psi (\a_1)=\a_1$, $\psi (\b_1)=\b_1$, $\psi(\a_i)=\d\tilde{\varphi}(\a_i)\d^{-1}$ and $\psi(\b_i)=\d\tilde{\varphi}(\b_i)\d^{-1}$,
 $i>1$.
\qed

\section{Appendix B: Loci and topological types}

In this appendix we show that, up to { essentially only one exception}, the  loci 
$M_{g,\rho}(D_n)$ are in bijection with the unmarked topological types
of $D_n$-actions.

In general, thanks to the work of of Singerman, Ries, and Magaard-Shaska-Shpectorov-V\"olklein (see \cite{singer}, \cite{ries} and \cite{mssv}), this is true more generally
for the irreducible components of the loci $M_g(G)$;  with a finite number of exceptions, they correspond bijectively to 
 the unmarked topological types
of $G$ actions. That there are indeed groups $G$ and  different unmarked topological types
of $G$ actions which yield { the same} component  was already shown by Ries (\cite{ries}).

To have a simple notation, assume in this appendix that $H, H'$ are distinct  finite subgroups of the mapping class group $Map_g$,
and denote by $Z : = Fix(H), Z' : = Fix(H')$ the corresponding irreducible analytic subsets of Teichm\"uller space $\sT_g$. We already observed that, if we denote
by $$\rho : H \ra Map_g, \rho' : H' \ra Map_g$$ the inclusion homomorphisms,  these analytic subsets map to irreducible closed algebraic 
sets $M_{g, \rho} (H)$, respectively $M_{g, \rho'} (H')$. 



We define the generic group of automorphisms of a curve in $Z$ as
$$ G : = G_H : =  \cap_{C \in Z} Stab_C  \  \ \ (Aut (C) \cong Stab_C 
\subset Map_g). $$


We refer to lemma 4.1 of \cite{mssv} for the proof of the following result. 

\begin{theo}\label{MSSV} {\bf (MSSV)} 
 Suppose $H \subset  G$ and $Z$ are as above, with $H$ a proper subgroup of $G$ and $C\in Z$.
Then
$$
\de : = dim (Z)  \leq 3.
$$

I) if  $\de = 3$, then $H$ has index $2$ in $G$, and $C \ra C/G$ is covering of $\PP^1$ branched on six points, $P_1, \dots, P_6$,
and with branching indices all equal to $2$.  Moreover the subgroup $H$ corresponds to the unique genus two double cover of $\PP^1$ 
branched on the six points, $P_1, \dots, P_6$ (by Galois theory, intermediate covers correspond to subgroups of $G$ bijectively).

II)  If $\de = 2$, then $H$ has index $2$ in $G$, and $C \ra C/G$ is covering of $\PP^1$ branched on five points, $P_1, \dots, P_5$,
and with branching indices  { $2,2,2,2,c_5$}.  Moreover the subgroup $H$ corresponds to a genus one double cover of $\PP^1$ 
branched on {  four  of the points  $P_1, \dots, P_4, P_5$ which have branching index $2$}.

III)  If $\de = 1$, then there are three possibilities.

III-a) $H$ has index $2$ in $G$, and $C \ra C/G$ is covering of $\PP^1$ branched on four points, $P_1, \dots, P_4$,
 with branching indices $2,2,2, 2 d_4$, where $d_4 > 1$.  Moreover the subgroup $H$ corresponds to the unique genus one double cover of $\PP^1$ 
branched on the four points, $P_1, \dots, P_4$.

III-b) $H$ has index $2$ in $G$, and $C \ra C/G$ is covering of $\PP^1$ branched on four points, $P_1, \dots, P_4$,
 with branching indices $2,2,c_3, c_4$, where $c_3 \leq c_4 > 2$.  Moreover the subgroup $H$ corresponds to a genus zero double cover of $\PP^1$ 
branched on two points whose branching index equals 2.

III-c) $H$ is normal  in $G$, $ G/H \cong (\ZZ/2)^2$, moreover  $C \ra C/G$ is covering of $\PP^1$ branched on four points, $P_1, \dots, P_4$,
 with branching indices $2,2,2, c_4$, where $ c_4 > 2$.  Moreover the subgroup $H$ corresponds to the unique  genus zero  cover of $\PP^1$ 
with group $(\ZZ/2)^2$ branched on the three  points $P_1, P_2, P_3$ whose branching index equals 2.

\end{theo}

\begin{cor}
\label{mssvCor}
Assume that
\begin{equation*}
\tag{$\ast$} Z : = Fix(H) = Z' : = Fix(H'),
\end{equation*}
where $ H \neq H'$ (  hence, in the previous notation,  $H$ is a proper subgroup of $G$). 

Then $\de: = dim (Z) \leq 2$.
{ Moreover, if $\de = 2$, then necessarily all the five branching indices are equal to $2$ ($c_5=2$)  }.

Assume further that $H, H'$ have the same cardinality.

 Then $ G : = G_H : =  \cap_{C \in Z} Stab_C$ induces a sequence of coverings $$ C \ra C/H \ra C/G$$  which is not of type
III-c), and , if it is type III-b), then the branching indices  $2,2,c_3, c_4$ must satisfy $c_3 = 2$.

If  $ C \ra C/H \ra C/G$   is  of type
III-a), then  $ C \ra C/H' \ra C/G$ is of type III-b) with branching indices  $ 2,2,2 ,  2 d_4$.

If moreover $\de =0 $, and $H \cong H' \cong D_n$, then necessarily $Z$ and $Z'$ are distinct points.

\end{cor} 

\Proof
We can  apply the previous theorem to both $H$ and $H'$, which are 
 proper subgroups of the same group $G$.

In the case I) where $\de = 3$, then both $H$ and $H'$ correspond to the unique genus two cover branched on the six points,
contradicting $ H \neq H'$.

{ The same argument applies in case II) when $c_5 > 2$.}

In the case where the covering $ C \ra C/H \ra C/G$ is of type III-c), then the index of $H$ is four and, since $ |H| = |H'|$, the index 
of $H'$ is also four, hence $H'$ corresponds to the unique genus zero $(\ZZ/2)^2$-cover branched on the three
points $P_1, P_2, P_3$; again we obtain the contradiction $H = H'$.

If the sequence of coverings $ C \ra C/H \ra C/G$   is  of type III-b), then the intermediate genus zero covering is unique unless
$ c_3= 2$.

If the type is III-a), the type for $H'$ cannot be the same, since we have a unique intermediate
genus one cover. Since $\de=1$, the index of $H'$ is also equal to two,  and the branching indices are  $ 2,2,2 ,  2 d_4$
the type of $H'$ must be III-b).

Finally, if $\de=0$ and $H \cong H' \cong D_n$, we have a dihedral covering of $\PP^1$ branched in three points, for which
there is a unique topological type (see \cite{CLP11}), corresponding to the monodromy factors $ y, y x, x^{-1}$.

\qed

\begin{rem}
In the case where $H \cong H' \cong D_n$, and $H \neq H'$ we have then $\de=1$ or $\de=2$.
Moreover, both $H, H'$ have index in $G$ equal to $2$.

If $\de = 2$, then $H, H'$ are both of type II), hence they correspond to intermediate genus one covers branched
on four of the five points $P_1, P_2,\dots,  P_5$.

If $\de = 1$, either  $H, H'$ are both of type III-b) (with branching indices $2,2,2,c_4$) hence they correspond to intermediate genus zero covers branched
on two of the three  points $P_1, P_2, P_3$; or, up to exchanging $H$ with $H'$, $H$  is of type III-a) and $H'$ is of type
 III-b) with branching indices $2,2,2, 2 d_4$.

\end{rem}

{ The investigation and detailed classification of these coincidences ($Z= Z'$ and $H \cong H' \cong D_n$, with $H \neq H'$)
is interesting, but shall not be  fully pursued here. We shall limit ourselves to show that it must be $\de = 1$,
and that there is only one exception, namely, the one where one group is of type III-a) and the other of type III-b).}
 
A common feature of all the cases is however the following situation. Defining $ K = H \cap H'$, we have an exact sequence
\begin{equation*}
\tag{$DD$}
1 \ra K \ra G \ra (\ZZ/2)^2 \ra 1.
\end{equation*}

\begin{prop}\label{DD}
Assume that  we have an exact sequence (DD) where $H \cong H' \cong D_n$.
Then there are  elements $\ga_1, \ga_2$ with $\ga_1 ^2 = \ga_2^2 = 1$, whose image  generate $(\ZZ/2)^2$, and such that
$$ H = \langle K , \ga_1 \rangle, \  H' = \langle K , \ga_2 \rangle.$$ 

Moreover  (DD) splits if and only if  $ G \cong   D_n  \times \ZZ/2$,  $H $ corresponds to the subgroup $ D_n \times \{ 0\}$, and $H'$ is the graph 
of a homomorphism $ \phi : D_n \ra \ZZ/2$ with Kernel equal to $K$.

If (DD) does not split, then $n$ is even and 
either $G \cong D_{2n}$ or $n = 4 h$, where $h$ is odd , and $G$ is the semidirect product of $H \cong D_n$ with $\langle \ga_2 \rangle \cong \ZZ/2$,
such that conjugation by $\ga_2$ acts as follows: 
$$  y \mapsto y x^2 , x \mapsto x^{2h-1}.$$ 
\end{prop}

\begin{rem}
\label{split}
The condition that (DD) splits always holds, except possibly if $H,H'$ are of different types III-a) and III-b).

Indeed if $H , H'\cong D_n$ are both of type II), then the Hurwitz generating systems of $C\to C/H$, $C\to C/H'$
must be of the form 
$$
v=(c_1,c_2;a,b) \, , \quad v'=(c_1',c_2',;a',b') \in HS(D_n;1,2) \, ,
$$
where $c_2$ is conjugate to $c_1$ (resp. $c'_2$ is conjugate to $c_1'$) in $G$
and $c_i^2=(c_i')^2=1$, $i=1,2$. From Theorem \ref{MSSV} it follows that $c_1,c_2,c_1',c_2' \not\in K=H\cap H'$,
therefore $c_1, c_2$ (resp. $c_1',c_2'$) are  reflections in $H$ (resp. $H'$).

Since  otherwise one of them
must be the unique central element of $H$ (resp. of $H'$); but the unique central element of $H$   is then central in $G$ ($H$ being normal in $G$), hence  it equals the unique central element of $H'$  by the same argument: hence it lies in $K$ and we derive a contradiction.

 From this we deduce that the numerical type of $v$ is the same as that of $v'$, up to 
automorphisms, and hence using Proposition \ref{general} it follows that $H$ and $H'$ have the same
unmarked topological type.

Finally, assume that $H, H' \cong D_n$ are both of type III-b). Let $v=(v_1,v_2,v_3,v_4)\in HS(G;0,4)$ be the Hurwitz generating system of $C\to C/G$.
By Corollary 7.2 we have:
$$
v_1^2 = v_2^2 = v_3^2 =1 \not= v_4^2 \, .
$$
Moreover, by Theorem \ref{MSSV},  $v_4\in H\cap H' =K$. Since  $ v_4^2 \neq 1$,  $v_4$ is a rotation in $H$ and in $H'$.
Now consider the Hurwitz systems of $C\to C/H$ and of $C\to C/H'$, which are of the following forms:
$$
u=(u_1, u_2, v_4, v_4') \, , \qquad u'=(u_1', u_2', v_4, v_4'') \, ,
$$
where $v_4'$ is conjugate to $v_4$ in $G$.
Notice that each of them contains only two reflections, hence by Remark 5.2 (R)
$v_4$ must generate the subgroup of rotations in $H$ (resp. in $H'$). But $v_4\in K$ and 
the subgroup of rotations of $K$ have index $2$ in the subgroup of rotations of $H$ (resp. of $H'$)
if $G$ is not the direct product of $D_n$ with $\mathbb{Z}/2$.

\end{rem}

\Proof{\em (of prop. \ref{DD}).}

$K$ is an index two subgroup in $D_n$. If $n$ is odd, then necessarily $ K = \sR$, the subgroup of rotations.

1) More generally, if $ K = \sR$, any element in $ H \setminus K$ has order two and splits the exact sequence
$   \ 1 \ra K \ra H \ra (\ZZ/2) \ra 1$, and the same holds for $H'$. Therefore we find elements $\ga_1, \ga_2$ as desired,
and such that their conjugation action on $K$ is the same, $ x^i \mapsto x^{-i}$.

Now, $\ga : = \ga_1 \ga_2$ centralizes $K = \sR$ and $\ga^2 \in K$. Let $x$ be  generator of $\sR$.
Hence
$$ \ga_1 \ga_2 \ga_1 \ga_2 = x^r .$$ 

Replacing $\ga_2 $ with $\ga_2 x^a$, we replace $r$ by $ r + 2a$. If $n$ is odd, then we may assume $r=0$,
whereas if $n$ is even, we also have the case $ r=1$. Observe that $r=0$ implies the splitting of the
above sequence (DD), moreover $\ga$ centralizes $K$.

 2) Assume that (DD) splits and $\ga = \ga_1 \ga_2$ centralizes $K$. 
 Then $$ \ga_1 \ga_2 \ga_1 \ga_2 = 1 \leftrightarrow \ga \ga_1 \ga = \ga_1 ,$$
 which amounts to $\ga$ being in the centre of $G$. Hence 
 $$ G \cong H \times \langle \ga \rangle \cong H \times  \ZZ/2.$$
 Moreover, $H' =  \langle K , \ga_2 \rangle =  \langle K , \ga_1 \ga \rangle$
 which proves our assertion.
 
 3) Assume now that $K = \sR$ and that (DD) does not split: then we can only achieve 
 $ \ga^2  = x .$ 
 Observe now that
 $$ \ga_1 \ga \ga_1 = \ga_2 \ga_1 = \ga^{-1} , \ga^{i} = 1 \leftrightarrow 2n | i .$$
 Since  $G$ is generated by $\ga, \ga_1$, it follows that $ G \cong D_{4m}$, $ n = 2m$. 
 
 4) Assume now that $n = 2m$ is even and $ K \neq \sR$. Then $ K \cong D_m$,
 and $K$ is generated, up to an automorphism, by $x^2$ and $y$. The element $z : = yx$
 has order two and its action by conjugation is
 $$  x^2 \mapsto x^{-2} , \ y \mapsto y x^2.$$ 
 
 In this way we construct also in this case the desired  elements $\ga_1, \ga_2$ such that 
 $\ga : = \ga_1 \ga_2$ centralizes $K$.
 
 Now, $ \ga^2$ is in the centre of $D_m$, hence  $ \ga^2 = 1$ if $m$ is odd,
 or, 
 if $ m = 2h$, then $ \ga^2 = x^{ 2 h} = x^m$.
 
 If $ \ga^2 = 1$ we have the splitting and we can apply 2).
 
 Otherwise, if $h$ is even, replace $\ga_2$ by $\ga_2 x^h$, which has again order equal to two.
 Then $\ga$ is replaced by $\xi : = \ga_1 \ga_2 x^h$. $\xi$ has order two,
 since
 $$ \xi^2 =   \ga_1 \ga_2 x^h  \ga_1 \ga_2 x^h =  \ga_1 \ga_2  \ga_1 \ga_2 x^{2h} =  \ga^2 x^m = x^{2m} = 1.$$
 
 Observe then that $  \ga_1 \ga_2 =  \ga_2 \ga_1 x^m$.
 
 We  show that then $\xi$ is in the centre:
 $$  \ga_1 \xi \ga_1 : = \ga_1  \ga_1 \ga_2 x^h \ga_1 = \ga_2 \ga_1 x^{-h} = \ga_1 \ga_2 x^{-m -h} = \ga_1 \ga_2 x^{h} = \xi,$$
  $$  \ga_2 \xi \ga_2 : = \ga_2  \ga_1 \ga_2 x^h \ga_2 = \ga_2 \ga_2 \ga_1 x^{m+h} \ga_2= \ga_1 x^{-h} \ga_2 = \ga_1 \ga_2 x^{h}
 = \xi.$$
 
 Assume instead that $h$ is odd. Then we observe simply that $G$ is generated by $H$ and by $\ga_2$, 
 and we set $x : = y \ga_1 \in H$, so that $H$ is generated in the standard way by $ y, x$.
 
 We have $\ga_2^2 = 1$, moreover conjugation by $\ga_2$ sends 
 $$  x^2 \mapsto x^{-2} , y \mapsto y x^{2}, \ga_1  \mapsto \ga_1 x^{2h} ,$$
 since $   \ga_1 \ga_2   \ga_1 \ga_2 = x^{2h} $.
 
 We conclude that $$x : = y \ga_1 \mapsto  y x^{2} \ga_1 x^{2h} = y x^{2} y x  x^{2h} = x^{2h-1}.$$

\qed 

\begin{prop}
\label{exclusion}
The case $\de=2$ cannot occur for the  group $D_n$.
\end{prop}

\Proof

We have already shown that in case II) 
the five branching indices must all be equal to $2$.

 By \ref{DD} and \ref{split}
the group $G$ generated by $H, H'$ must be $G = D_n \times  C_2$, where $C_2 : = \ZZ/2$,
$H$  is the subgroup $D_n \times  \{0\}$ , while $H' $ is the graph of  
 a homomorphism of $H$ onto $C_2$.

We consider first the case where four of the five elements have a
reflexion component. To have $(y,0)$ in the  group they generate, the second
component of one of these four elements must be trivial. 
We permute the other three into the first three positions, apply Lemma
2.1 of \cite{CLP11} to make the first two equal. 
So after a suitable automorphism of $D_n$ the tuple is
$$
(y,1)(y,1)(yx^\ell,1)(yx^{\ell+m},0)(x^m,1)
$$
where $m$ is an integer in $\{0,\frac n2\}$.

In case $n$ even and $m=n/2$ we may apply the following
automorphism of $G$, $(x,0)\mapsto (x,0)$, $(y,0)\mapsto (yx^m,0)$,
$(e,1)\mapsto (x^m,1)$.

So we are left with the case $m=0$, where we can apply another
automorphism to get
$$
(y,1)(y,1)(yx,1)(yx,0)(e,1).
$$
Since only one of these elements is in $H'$, and $ H' \neq H$, $H'$ is the
graph of the homomorphism which sends $y$ to $0$ and $x$ to $1$,
in particular $n$ must be even.

On the two intermediate covers, which are elliptic curves, we get the
respective Hurwitz vectors $(yx,yx; x,1)$ and $(yx',yx'; x',1)$, where
$x'$ is a shorthand for $(x,1)\in H'$. Obviously they are the same
under the automorphism of $G$ which keeps $(y,0)$ and $(y,1)$ fixed
and sends $(x,0)$ to $(x,1)$. This means that $H = H'$, a contradiction.

The analysis in the second case, where two of the five elements have a
reflection component, is similar. Up to automorphisms and braid
equivalence, there is only one possible tuple and $n$ must equal $2$:
$$
(x,1)(0,1), (x,1)(y,1)(y,0).
$$
Here $H'$ must be the
graph of the homomorphism which sends $y$ to $1$ and $x$ to $0$.
And on the two intermediate covers we get the
respective Hurwitz vectors $(y,y; x,x)$ and $(y',y'; x,x)$, 
where $y'$ is a shorthand for $(y,1)\in H'$. They are in one orbit
under the automorphism of $G$ which keeps $(x,0)$ and $(x,1)$ fix
and which exchanges $(y,0)$ with $(y,1)$.

\qed

\begin{theo}\label{dimension}
Assume that we have two distinct subgroups of $Map_g$, $H, H' \cong D_n$, and that  $Z : = Fix(H) = Fix(H')$.
Then $\de: = dim (Z) = 1$, and {  the numerical invariant $g'$ cannot be the same for both actions.}
\end{theo}

\Proof
We have shown in cor. 7.2 that $ \de \leq 2, \de \neq 0$, and in proposition 7.6 that $\de \neq 2$. Hence $\de=1$.

By remark 7.3, if we make the assumption  
that $g'$ is the same, follows then that $H$ and $H'$ are both of type III-b). By prop. 7.4 and remark 7.5, the group is then  $G = D_n \times  C_2$,

$H$ is $D_n \times \{0\}$ and $H'$ is the graph of some homomorphism from $D_n$ to $C_2$.
By the last paragraph of Remark 7.5, $K$ coincides with the group of rotations of $H=D_n\times \{0\}$, 
hence $H'=\ker(f)$, where $f: G=D_n \times C_2 \to C_2$ is given by 
$f(y^bx^i, a)= a+ b$.

By remark 7.3, the branching indices of $C\to C/G$ are $2,2,2,c_4$, therefore up to a permutation we get

the sequences
$$ 
a = 1,1,0,0, \ 
a+b = 0,1,1,0 , \ 
b= 1,0,1,0.
$$
 
Hence we get that the first coordinates of the second and of the  last element are  rotations of respective orders $d_2$, $c_4$, where $d_2 = 1$ or $d_2 = 2$.

There is  surjection $ G \ra D_n$ with the second element in the kernel. Now, if $D_n$ is generated by two reflections $\sigma_1, \sigma_2$ then the rotation $\sigma_1 \sigma_2$ 
generates the group $\sR$ of rotations. We conclude that the order $c_4 = n$.

We may assume, up to an automorphism of $D_n$,  that the last element is $(x,0)$, and the first one is $(y,1)$.

Then the 4-tuple is:
\begin{equation*}
(y,1) (x^{h},1)(yx^{h-1},0)(x,0),
\end{equation*}
where $ h =0$ or  $ h = n/2$.

The respective subgroups are generated by $\ga_3$, $\ga_4$, $\ga_1 \ga_3 \ga_1= : \ga_3'$,  $\ga_1 \ga_4 \ga_1= : \ga_4'$
in one case,
by  $\ga_4$, $\ga_1^* : = \ga_1^{-1}$, $\ga_2 \ga_4 \ga_2$,  $\ga_2 \ga_1^* \ga_2$ in the other.

We get the two corresponding length 4 Hurwitz vectors.
$$
(yx^{h-1}, x, yx^{h + 1}, x^{-1}) , \ 
(x, y, x, y).
$$
The corresponding two Nielsen functions are { equal for $n$ odd ($h=0$), distinct for $n$ even,
but in the same orbit under the group $ \Aut (D_n)$.

By \cite{CLP11}, Theorem 2, $H, H'$ correspond then to the same unmarked topological type.  }

\qed

\begin{theo}\label{exception}
Assume that we have two distinct subgroups of $Map_g$, $H, H' \cong D_n$, and that  $Z : = Fix(H) = Fix(H')$.
Then $\de: = dim (Z) = 1$, and case III-a) holds for $H$,  case III-b) holds for $H'$.

This case actually occurs.
\end{theo}

\Proof
That the only possible case is the one described follows by theorem \ref{dimension} and the previous discussion.
 It suffices thus
to give a concrete example.

We have  a polygonal group $T(2,2,2,2d)$, generated by elements
$\ga_1$, $\ga_2$, $\ga_3$, $\ga_4$ 
whose orders are respectively $2,2,2,2d$ and
whose product $\ga_1 \ga_2 \ga_3 \ga_4 $ is the identity.

The elliptic induced covering is generated by 
$$
a : = \ga_1 \ga_2,\quad
b : = \ga_2 \ga_3,\quad  
c : = (\ga_1 \ga_4 \ga_1 )^2
$$
which satisfy the usual presentation $[a,b] = c $.

Whereas for the genus zero induced covering, since we have a double cover branched on the first two points, we get generators
$$
\ga_3,\quad 
\ga_4,\quad 
\ga'_3 : = \ga_1 \ga_3 \ga_1,\quad 
\ga '_4 : = \ga_1 \ga_4 \ga_1,
$$
which satisfy the usual presentation 
$$
\ga_3 \ga_4  \ga'_3  \ga '_4= 1.
$$

These elements have respective orders $2, 2d, 2, 2d$, so we guess that $\ga_4$ should map to  a generator of the rotation group, say $\ga_4 
\mapsto  x$, and $\ga_3 $
should map to  a reflection, say $\ga_3 \mapsto y$.

We make the following assumption: since $\ga_1$  normalizes $H'$, let us just assume that  $\ga_1$  centralizes $H'$.

Then  we have a direct product
$G = H' \times C_2 \cong D_n \times  C_2 $, where  $n = 2d $, and the cyclic group  $C_2$  of order $2$is generated by $\ga_1$.

We take as Hurwitz vector  for $G$  (images of $\ga_1, \ga_2, \ga_3, \ga_4 $)
$$
(0,1) (yx, 1) (y,0) (x,0).
$$ 
Clearly these four elements  generate then $G$.

Now  $\ga_3 , \ga_4 ,  \ga'_3 , \ga '_4$
are respectively sent to   $(y,0)(x,0) (y,0)(x,0)$, and they generate the group $H' \cong D_n$.

 On the other hand  
$$
a \mapsto   (yx, 0),\quad 
b \mapsto   (x^{-1}, 1),\quad 
c \mapsto   (x^2, 0)
$$
and clearly $ [a,b]   \mapsto   (x^2, 0)$.

Moreover, the image group $H$ projects, under the first coordinate, onto the dihedral group $ D_n.$

\qed

{\bf Acknowledgement:}
\\

The authors would like to thank Binru Li and Sascha Weigl for
providing the argument which helped us to fix  a gap in a previous proof of 
Prop.\ref{exclusion}.

\bigskip
\noindent {\bf Authors' Address:}\\
\noindent Fabrizio Catanese, Michael L\"onne\\ Lehrstuhl Mathematik VIII,\\ 
Mathematisches
Institut der Universit\"at
Bayreuth\\ NW II,  Universit\"atsstr. 30\\ 95447 Bayreuth\\

Fabio Perroni, \\
SISSA - International School for Advanced Studies, School of Mathematics\\
 via Bonomea, 265 - 34136 Trieste (Italy).\\
 
          email:       
          
             fabrizio.catanese@uni-bayreuth.de,
          
michael.loenne@uni-bayreuth.de,
 
perroni@sissa.it.

\end{document}